\documentclass[11pt,a4paper]{article}
\usepackage{pict2e}
\usepackage{booktabs}
\usepackage{color}
\usepackage{epsfig}
\usepackage{epstopdf}
\usepackage{longtable}
\usepackage[T1]{fontenc}
\usepackage{amsmath}
\usepackage{cases}
\usepackage{geometry}
\usepackage{amsbsy,latexsym,amsfonts, epsfig, color, authblk, amssymb, graphics, bm}
\usepackage{epsf,slidesec,epic,eepic}
\usepackage{fancybox}
\usepackage{fancyhdr}
\usepackage{setspace}
\usepackage{nccmath}
\usepackage{diagbox}
\usepackage[colorlinks, citecolor=blue]{hyperref}

\newtheorem{theorem}{Theorem}[section]
\newtheorem{lemma}{Lemma}[section]

\newtheorem{proposition}{Proposition}[section]

\newtheorem{remark}{Remark}[section]

\newenvironment{proof}{{\noindent \bf Proof:}}{\hfill$\Box$\medskip}

\definecolor{lred}{rgb}{1,0.8,0.8}
\definecolor{lblue}{rgb}{0.8,0.8,1}
\definecolor{dred}{rgb}{0.6,0,0}
\definecolor{dblue}{rgb}{0,0,0.5}
\definecolor{dgreen}{rgb}{0,0.5,0.5}

 \title{Weighted iteration complexity of the sPADMM on the KKT residuals for convex composite optimization}

\author{Li Shen\footnote{Department of Mathematics, South China University of Technology, Guangzhou, 510641, China (shen.li@mail.scut.edu.cn).}
 \ \ {\rm and}\ \ Shaohua Pan\footnote{Corresponding author. Department of Mathematics, South China University of Technology, Tianhe District of Guangzhou City, China (shhpan@scut.edu.cn).}}

 \date{April 18, 2016}

\begin{document}

 \maketitle

 \begin{abstract}

  In this paper we establish an $\mathcal{O}({1}/{k})$ weighted iteration complexity on
  the KKT residuals yielded by the sPADMM (semi-proximal alternating direction method of multiplier)
  for the convex composite optimization problem. This result, which is derived
  with the help of a novel generalized HPE (hybrid proximal extra-gradient)
  iteration formula, first fills the gap on the ergodic iteration complexity of
  the classic ADMM with a large step-size and its many proximal variants.

 \end{abstract}

 \section{Introduction}\label{sec1}

  Let $\mathbb{Y},\mathbb{Z}$ and $\mathbb{X}$ be three real finite-dimensional Euclidean spaces
  which are equipped with an inner product $\langle\cdot,\cdot\rangle$ and its induced norm $\|\cdot\|$.
  We shall study the weighted iteration complexity of the ADMM for the following
  convex composite optimization problem
 \begin{equation}\label{prob}
  \min_{y\in\mathbb{Y},z\in\mathbb{Z}}\Big\{\vartheta(y)+f(y)+\varphi(z)+g(z)\!:\ \mathcal{A}^*y+ \mathcal{B}^*z =c\Big\},
 \end{equation}
 where $\vartheta\!:\mathbb{Y}\to\!(-\infty,+\infty]$ and $\varphi\!:\mathbb{Z}\to\!(-\infty,+\infty]$
 are two closed proper convex functions, $f\!:\mathbb{Y}\to\!(-\infty,+\infty)$ and
 $g\!:\mathbb{Z}\to\!(-\infty,+\infty)$ are two continuously differentiable convex functions,
 $\mathcal{A}^*\!: \mathbb{Y}\to\mathbb{X}$ and $\mathcal{B}^*\!: \mathbb{Z}\to\mathbb{X}$ are
 the adjoints of the linear operators $\mathcal{A}\!: \mathbb{X}\to\mathbb{Y}$
 and $\mathcal{B}\!: \mathbb{X}\to\mathbb{Z}$, respectively, and $c\in\mathbb{X}$ is a given point.
 A typical example of \eqref{prob} is the Lagrange dual of the convex
 composite quadratic semidefinite programming (see \cite{HSZ15}):
 \begin{align}\label{quadratic-DSDP}
  &\min\ \delta_{\mathbb{S}_{+}^p}(s)-\langle b,y\rangle+\frac{1}{2}\langle w,\mathcal{Q}w\rangle+\phi^*(-z)\nonumber\\
  &\  {\rm s.t.}\ \ s+\mathcal{A}^*y-\mathcal{Q}w+z=c,\ w\in\mathcal{W}
 \end{align}
 where $\mathbb{S}_{+}^p$ is the cone consisting of all $p\times p$ positive semidefinite matrices
 in $\mathbb{S}^p$, the space of all $p\times p$ real symmetric matrices,
 $\delta_{\mathbb{S}_{+}^p}(\cdot)$ is the indicator function over $\mathbb{S}_{+}^p$,
 $\mathcal{Q}\!:\mathbb{S}^p\to\mathbb{S}^p$ is a self-adjoint positive semidefinite linear operator,
 $\phi\!:\mathbb{S}^p\to(-\infty,+\infty]$ is a closed proper convex function,
 $\mathcal{A}\!:\mathbb{S}^p\to\mathbb{R}^m$ is a linear operator,
 and $\mathcal{W}$ is any linear subspace in $\mathbb{S}^p$ containing the range of $\mathcal{Q}$.
 Clearly, problem \eqref{quadratic-DSDP} has the form of \eqref{prob} when viewing
 $(s,y)\in\mathbb{S}^p\times\mathbb{R}^m$ and $(w,z)\in\mathbb{S}^p\times\mathbb{S}^p$
 in \eqref{quadratic-DSDP} as the variables $y$ and $z$ in \eqref{prob}, respectively.
 For the applications of the convex composite quadratic semidefinite programming,
 the interested readers may refer to \cite{AKW99,Higham02,QiS06,Toh08}.
 Another interesting example of \eqref{prob} is the Lagrange dual of the nuclear semi-norm
 penalized least squares model which was proposed in \cite{MPS15} for resolving
 the matrix completion with fixed basis coefficients:
 \begin{align}\label{quadratic-DSDP}
  &\min\ \delta_{\mathcal{K}}(-Z)+\frac{1}{2}\|\xi\|^2+\langle d,\xi\rangle-\langle b,y\rangle+\delta_{\mathbb{B}_{\rho}}(S)\nonumber\\
  &\  {\rm s.t.}\ \ Z+\mathcal{B}^*\xi+S+\mathcal{A}^*y=-\rho C
 \end{align}
 where $\delta_{\mathcal{K}}(\cdot)$ and $\delta_{\mathbb{B}_{\rho}}(\cdot)$ are
 the indicator functions over the polyhedral cone $K$ and the ball of spectral norm
 $\mathbb{B}_{\rho}=\{X\in\mathbb{R}^{m\times n}\ |\ \|X\|\le \rho\}$,
 $\mathcal{A}^*\!:\mathbb{R}^p\to\mathbb{R}^{m\times n}$
 and $\mathcal{B}^*\!:\mathbb{R}^m\to\mathbb{R}^{m\times n}$ are two given linear mappings,
 and $d\in\mathbb{R}^p, b\in\mathbb{R}^m$ and $C\in\mathbb{R}^{m\times n}$ are given data.

 \medskip

 The classic ADMM was originally proposed by Glowinski and Marroco \cite{GM75} and
 Gabay and Mercier \cite{GM76}, and its design was much inspired by
 Rockafellar's works on proximal point algorithms for the general maximal monotone
 inclusion problems \cite{Roc76a,Roc76b}. The readers may refer to Glowinski \cite{G14}
 for a note on the historical development of the classic ADMM, to Gabay and Mercier
 \cite{GM76} and Fortin and Glowinski \cite{GF83} for its convergence analysis
 under certain settings, and to Eckstein and Yao \cite{EY15} for a recent survey on this.

 \medskip

 This work focuses on the weighted iteration complexity of the ADMM, which will be conducted
 under a more convenient semi-proximal ADMM setting proposed by Fazel et al. \cite{FPST13}
 by allowing the dual step-size to be at least as large as the golden ratio $1.618$.
 This proximal ADMM, which provides a unified framework for the classic ADMM and
 its many variants, has not only the advantage to resolve the potentially non-solvability
 issue of the subproblems involved in the classical ADMM, but more importantly the ability
 to handle the multi-block separable convex optimization problems.
 For example, the symmetric Gauss-Seidel (sGS) based ADMM proposed by Li et al.
 \cite{LST16,STY14} for the separable convex minimization involving two nonsmooth convex
 and multiple linear/quadratic functions has been shown to be a sPADMM with
 a special proximal term.

 \medskip

 For any self-adjoint positive semidefinite linear operator $\mathcal{E}\!: \mathbb{H}\to\mathbb{H}$,
 where $\mathbb{H}$ is another real finite-dimensional Euclidean space equipped with an inner product
 $\langle\cdot,\cdot\rangle$ and its induced norm $\|\cdot\|$,
 we define $\|w\|_{\mathcal{E}}:=\sqrt{\langle w,\mathcal{E}w\rangle}$,
 ${\rm dist}_{\mathcal{E}}(w,C):=\inf_{w'\in C}\|w'-w\|_{\mathcal{E}}$
 and $D_{\mathcal{E}}(w,w'):=\frac{1}{2}\|w-w'\|_{\mathcal{E}}^2$ for any $w,w'\in\mathbb{H}$
 and any given closed set $C\subseteq\mathbb{H}$.  For any convex function $\phi\!:\mathbb{H}\to(-\infty,+\infty]$,
 we denote its effective domain by ${\rm dom}\,\phi:=\{w\in\mathbb{H}\ |\ \phi(w)<+\infty\}$ .
 Write $\vartheta_{\!f}(\cdot)\equiv \vartheta(\cdot)+f(\cdot)$
 and $\varphi_g(\cdot)\equiv \varphi(\cdot)+g(\cdot)$. Let $\beta>0$ be a given parameter.
 The augmented Lagrangian function of problem \eqref{prob} is defined by
 \[
   L_{\beta}(y,z;x):=\vartheta_{\!f}(y)+\varphi_g(z)+\langle x,\mathcal{A}^*y +\mathcal{B}^*z-c\rangle
   +\frac{\beta}{2}\|\mathcal{A}^*y +\mathcal{B}^*z-c\|^2.
 \]
 The sPADMM proposed in \cite{FPST13} generates the sequence $\{(y^k,z^k,x^k)\}_{k\ge 1}$ from
 an initial point $(y^0,z^0,x^0)\in{\rm dom}\,\vartheta\times{\rm dom}\,\varphi\times\mathbb{X}$
 by the following iteration formula
 \begin{subnumcases}{}\label{y-sPADMM}
   y^{k+1}\in\mathop{\arg\min}_{y\in\mathbb{Y}}\Big\{L_{\beta}(y,z^k;x^k)+\frac{1}{2}\|y\!-\!y^{k}\|^2_{\mathcal{S}}\Big\},\\
  \label{z-sPADMM}
   z^{k+1}\in\mathop{\arg\min}_{z\in\mathbb{Z}}\Big\{L_{\beta}(z^{k+1},y;x^k)+\frac{1}{2}\|z\!-\!z^{k}\|^2_{\mathcal{T}}\Big\},\\
  \label{x-sPADMM}
  x^{k+1}=x^{k} +\tau\beta\big(\mathcal{A}^*y^{k+1}+\mathcal{B}^*z^{k+1}-c\big),
 \end{subnumcases}
 where $\tau>0$ is a parameter to control the dual step-size, and $\mathcal{S}\!:\mathbb{Y}\to\mathbb{Y}$
 and $\mathcal{T}\!:\mathbb{Z}\to\mathbb{Z}$ are the given self-adjoint positive semidefinite linear operators.
 When $\mathcal{S}=\mathcal{T}=0$ and $\tau\in(0,\frac{1\!+\!\sqrt{5}}{2})$,
 equation \eqref{y-sPADMM}-\eqref{x-sPADMM} recovers the iterations of
 the classic ADMM. When $\mathcal{S}$ and $\mathcal{T}$ are positive definite,
 the above iteration scheme was initiated by Eckstein \cite{Eckstein94} to make
 the subproblems \eqref{y-sPADMM} and \eqref{z-sPADMM} easier to solve.
 While for positive semidefinite $\mathcal{S}$ and $\mathcal{T}$, Fazel et al. \cite{FPST13}
 developed an easy-to-use convergence theorem by using the essentially same
 variational techniques as in \cite{G08}, which plays a key role in the convergence
 analysis of the sGS based semi-proximal ADMM \cite{LST161}.
 In this paper, we assume that the self-adjoint positive semidefinite
 $\mathcal{S}$ and $\mathcal{T}$ are chosen such that
 (\ref{y-sPADMM}) and (\ref{z-sPADMM}) are well defined.

 \medskip

 For the sPADMM, Han, Sun and Zhang \cite{HSZ15} recently derived the local linear
 convergence rate under an error bound condition,
 while in this work we focus on an important global property, i.e.,
 the ergodic iteration complexity. The idea of considering averages of the iterates
 is very popular in the analysis of gradient-type and/or proximal point based methods
 for convex minimization and monotone variational inequalities (see, e.g.,
 \cite{NemY78, Nem04, Nest09, Lan12,CP15,RS2010,NedL15}). The existing ergodic iteration
 complexity of the ADMM are mostly derived for the objective values of the optimization problems.
 For example, Davis and Yin \cite{DY14} derived an $\mathcal{O}(1/k)$ ergodic iteration
 complexity on the objective values of problem \eqref{prob} for several splitting algorithms,
 including the classic ADMM with $\tau=1$ as a special case,
 Shefi and Teboulle \cite{ST14} conducted a comprehensive study on the iteration
 complexities, in particular in the ergodic sense for the sPADMM with $\tau=1$,
 Cui et al. \cite{CLST16} established an ergodic iteration complexity of the same order
 for the sPADMM with $\tau\in(0,\frac{1\!+\!\sqrt{5}}{2})$ where
 the objective function is allowed to have a coupled smooth term,
 and Ouyang et al. \cite{OCLJ15} provided an ergodic iteration complexity
 of the same order on the objective values for an accelerated linearized ADMM.
 To the best of our knowledge, the first ergodic iteration complexity of the ADMM on
 the KKT residuals instead of the objective values was derived by Monteiro and
 Svaiter \cite{RS2013} without requiring the boundedness of the feasible set.
 There the authors applied the weighted iteration complexity of the HPE method
 \cite{SS99,SS01} established in \cite{RS2010} for the maximal monotone
 operator inclusion problems to a block-decomposition algorithm and obtained
 an $\mathcal{O}({1}/{k})$ ergodic iteration complexity for the classic ADMM
 with the unit step-size $\tau=1$.

 \medskip

 Although there are so many research works on the iteration complexity of the ADMM,
 there is no ergodic iteration complexity on the KKT residuals even for the classic ADMM
 with a large step-size $\tau\in[1,\frac{1\!+\!\sqrt{5}}{2})$.
 As illustrated by the numerical results in \cite{WGY10,STY14,LST161}, the ADMM with a large
 step-size can improve in many cases at least $20\%$ the performance of the ADMM with a unit step-size.
 Then, it becomes an open issue whether the ADMM with a large step-size has an $\mathcal{O}(1/k)$
 ergodic iteration complexity on the KKT residuals or not. The contribution of this work
 is to resolve this open issue by establishing an $\mathcal{O}({1}/{k})$ weighted iteration
 complexity on the KKT residuals for the SPADMM with $\tau\in(0,\frac{1\!+\!\sqrt{5}}{2})$,
 and obtain the weighted iteration complexity on the KKT residuals for the sGS based ADMM as a by-product.
 This result was achieved with the help of a novel generalized HPE iteration formula
 (see Proposition \ref{ADMM-ergodic} and Remark \ref{remark-sequence}).

 \section{Main result}\label{sec3}

 The Karush-Kuhn-Tucker (KKT) system for problem \eqref{prob} takes the following form
 \begin{equation}\label{KKT}
  0\in \partial\vartheta(y)+\nabla f(y)+\mathcal{A}x,\
  0\in \partial\varphi (z)+\nabla g(z)+\mathcal{B}x,\
  \mathcal{A}^*y+\mathcal{B}^*z=c.
 \end{equation}
 We assume that there exists a point, say $(y^*,z^*,x^*)$, satisfying the KKT system \eqref{KKT}.
 Write $\mathbb{H}:=\mathbb{Y}\times\mathbb{Z}\times\mathbb{X}$ and
 let $T\!:\mathbb{H}\rightrightarrows\mathbb{H}$ be the operator defined by
 \begin{equation}\label{ADMM-Toper}
   T(w):=\left[\begin{matrix}
      \partial \vartheta_{\!f}(y)+\mathcal{A}x \\
      \partial\varphi_{g}(z)+\mathcal{B}x\\
       c-\mathcal{A}^*y-\mathcal{B}^*z \\
       \end{matrix}\right]
      \quad\ \forall\,w=(y,z,x)\in\mathbb{H}.
 \end{equation}
 With such $T$, the KKT system \eqref{KKT} can be reformulated as the inclusion problem $0\in T(w)$.
 Since both $\partial \vartheta_{\!f}$ and $\partial\varphi_{g}$ are maximally monotone (see \cite{Roc70}),
 there exist two self-adjoint and positive semidefinite linear operators $\Sigma_{\vartheta_{\!f}}$
 and $\Sigma_{\varphi_{g}}$ such that for all $y',y\in{\rm dom}\,\vartheta_{\!f}$,
 $\xi'\in\partial \vartheta_{\!f}(y')$ and $\xi\in\partial\vartheta_{\!f}(y)$,
 and for all $z',z\in{\rm dom}\,\varphi_{g}$,
 $\eta'\in\partial\varphi_{g}(z')$ and $\eta\in\partial\varphi_{g}(z)$,
 \begin{equation}\label{fgmonotone}
   \langle \xi'\!-\!\xi,y'\!-\!y\rangle\ge\|y'\!-\!y\|_{\Sigma_{\vartheta_{\!f}}}^2\ \ {\rm and}\ \
   \langle \eta'\!-\!\eta,z'\!-\!z\rangle\ge\|z'\!-\!z\|_{\Sigma_{\varphi_{g}}}^2.
\end{equation}
 From equation \eqref{fgmonotone} and the definition of $T$, for all $w',w\in{\rm dom}\,T$,
 $\zeta'\in T(w')$ and $\zeta\in T(w)$,
 \begin{equation}\label{Tmonotone}
  \langle \zeta'\!-\!\zeta,w'\!-\!w\rangle\ge\|w'\!-\!w\|_{\Sigma}^2\quad\ {\rm with}\ \
  \Sigma:={\rm Diag}\left(\begin{matrix}\Sigma_{\vartheta_{\!f}},\Sigma_{\varphi_{g}},0\end{matrix}\right).
 \end{equation}
 This implies that $T$ is monotone. In fact, one may verify that $T$ is maximally monotone.

 \medskip

 In the rest of this paper, we let $\mathcal{G}\!:\mathbb{H}\to\mathbb{H}$ and
 $\mathcal{M}\!:\mathbb{H}\to\mathbb{H}$ be the self-adjoint block diagonal
 positive semidefinite linear operators defined by
 \begin{align}\label{Goper}
  \mathcal{G}:={\rm Diag}\left(\begin{matrix}
                \mathcal{S},\,\mathcal{T}\!+\!\beta\mathcal{B}\mathcal{B}^*,\,(\tau\beta)^{-1}\mathcal{I}
                 \end{matrix}\right),\qquad\quad\\
  \mathcal{M}:={\rm Diag}\left(\begin{matrix}
                \mathcal{S}\!+\!\Sigma_{\vartheta_{\!f}},\,\mathcal{T}\!+\!\beta\mathcal{B}\mathcal{B}^*\!+\!\Sigma_{\varphi_{g}},
                \,(\tau\beta)^{-1}\mathcal{I}
                 \end{matrix}\right).
  \label{Moper}
 \end{align}
 By the definition of $\Sigma$ in \eqref{Tmonotone}, clearly, $\mathcal{M}=\mathcal{G}+\Sigma$.
 For any $\tau\in (0,\frac{1\!+\!\sqrt{5}}{2})$, we write
 \begin{subnumcases}{}\label{sigma-def}
 \sigma_{\!\tau}:=\!\max\!\Big\{\frac{2-\tau}{2},\frac{1+\tau(\tau\!-\!1)^2}{2-(1\!-\!\tau)^2}\Big\},\\
 \label{gamma-def}
 \gamma_{\tau}:=\min\left\{\frac{\min(\tau^2\sigma_{\!\tau},\sigma_{\!\tau}\!-\!(\tau\!+\!1)(\tau\!-\!1)^2)}{\tau(\sigma_{\!\tau}\!-\!(\tau\!-\!1)^2)},
   \frac{\tau\sigma_{\!\tau}}{\sigma_{\!\tau}(\tau\!+\!1)\!+\!\tau\!-\!1}\right\},\\
  \nu_{\!\tau}\!:=\sigma_{\!\tau}(\tau\!+\!1)\!+\!\tau\!-\!1.
 \end{subnumcases}
 By using the definitions of $\sigma_{\!\tau}$ and $\nu_{\tau}$, it is not difficult to verify that
 \begin{align}\label{sigma-h1h2-equa1}
   \sigma_{\!\tau}\in(0,1)\ \ {\rm and}\ \ \nu_{\tau}\ge\max(\tau\sigma_{\!\tau},\sigma_{\!\tau}\!-\!(\tau\!-\!1)^2)\ \ {\rm for}\ \tau\in\Big(0,\frac{1\!+\!\sqrt{5}}{2}\Big),\qquad\\
  \label{sigma-h1h2-equa2}
  \tau^2\sigma_{\!\tau}<\tau(\sigma_{\!\tau}-(\tau\!-\!1)^2)<\sigma_{\!\tau}\!-\!(\tau\!+\!1)(\tau\!-\!1)^2\ \ {\rm for}\ \tau\in(0,1),\qquad\qquad\\
  \tau^2\sigma_{\!\tau}>\tau(\sigma_{\!\tau}\!-\!(\tau\!-\!1)^2) >\sigma_{\!\tau}-(\tau+1)(\tau\!-\!1)^2>0\ \ {\rm for}\ \tau\in\Big[1,\frac{1\!+\!\sqrt{5}}{2}\Big).\quad
 \label{sigma-h1h2-equa3}
 \end{align}
 Clearly, $\gamma_{\tau}\in(0,1)$ for $\tau\in(0,\frac{1+\sqrt{5}}{2})$ by \eqref{sigma-h1h2-equa1}-\eqref{sigma-h1h2-equa3}.
 Also, the definition of $\sigma_{\!\tau}$ implies that
 \begin{equation}\label{imply-sigma}
   \sigma_{\!\tau}(\tau\!+\!1)\!-\!1-(\sigma_{\!\tau}+\tau\!-\!1)\max(1\!-\!\tau,(\tau\!-\!1)\tau)=0.
 \end{equation}
 In addition, for $k=1,2,\ldots$, we denote $w^{k}:=(y^{k},z^{k},x^{k})$,
  $\widetilde{w}^k:=(y^{k},z^{k},\widetilde{x}^{k})$ with
  \(
    \widetilde{x}^{k}\!:=x^{k-1}\!+\!\beta(\mathcal{A}^*y^{k}\!+\!\mathcal{B}^*z^{k-1}\!-\!c),
  \)
 $r^{k}:=\mathcal{G}(w^{k-1}\!-\!w^{k})$, and $\eta_k:=\widetilde{\eta}_k+\frac{1}{4}\|w^k\!-\!w^{k-1}\|_{\Sigma}^2$ with
 \begin{align}
  \widetilde{\eta}_{k}&:=\frac{(\sigma_{\!\tau}\!-\!(\tau\!-\!1)^2)\beta}{2\tau}\|\mathcal{A}^*y^{k}\!+\!\mathcal{B}^*z^{k}\!-\!c\|^2
                         +\frac{\sigma_{\!\tau}}{2}\|y^{k}\!-\!y^{k-1}\|^2_{\mathcal{S}}\nonumber\\
                      &\quad\ +\frac{\nu_{\!\tau}}{2\tau}\|z^{k}\!-\!z^{k-1}\|^2_{\mathcal{T}}
                             +\frac{\sigma_{\!\tau}\!+\!\tau\!-\!1}{\tau}\|z^{k}\!-\!z^{k-1}\|^2_{\Sigma_{\varphi_{g}}}.
 \end{align}

 Before achieving the main result, we prove some properties of the sequence $\{w^k\}$.
 \begin{lemma}\label{lemma1-sequence}
  Let $\{(y^{k},z^{k},x^{k})\}$ be the sequence generated by (\ref{y-sPADMM})-(\ref{x-sPADMM})
  with $\tau\in(0,+\infty)$. Then, for all $k\ge 1$, it holds that
  $\mathcal{G}(w^{k}\!-\!w^{k+1})\in T(\widetilde{w}^{k+1})$.
 \end{lemma}
 \begin{proof}
   By the optimality of $y^{k+1}$ and $z^{k+1}$ in equations (\ref{y-sPADMM}) and (\ref{z-sPADMM}), we know that
  \begin{subnumcases}{}
   0\in \partial\vartheta_{\!f}(y^{k+1})+\mathcal{A}x^{k}+\beta\mathcal{A}(\mathcal{A}^*y^{k+1}+\mathcal{B}^*z^{k}-c)+\mathcal{S}(y^{k+1}-y^{k}), \label{opt-x}\\
   0\in \partial\varphi_{g}(z^{k+1})+\mathcal{B}x^{k}+\beta\mathcal{B}(\mathcal{A}^*y^{k+1}+\mathcal{B}^*z^{k+1}-c)+\mathcal{T}(z^{k+1}-z^{k}).
   \label{opt-y}
  \end{subnumcases}
  Substituting $x^{k}=\widetilde{x}^{k+1}-\beta(\mathcal{A}^*y^{k+1}+\mathcal{B}^*z^{k}-c)$
  into (\ref{opt-x})-(\ref{opt-y}) respectively yields that
  \begin{subnumcases}{}
   0\in \partial\vartheta_{\!f}(y^{k+1})+\mathcal{A}\widetilde{x}^{k+1}+\mathcal{S}(y^{k+1}-y^{k}),\nonumber\\
   0\in \partial\varphi_{g}(z^{k+1})+\mathcal{B}\widetilde{x}^{k+1}+(\mathcal{T}\!+\!\beta\mathcal{B}\mathcal{B}^*)(z^{k+1}-z^{k}).\nonumber
  \end{subnumcases}
 Together with the equality $x^{k+1}-x^{k}=\tau\beta(\mathcal{A}^*y^{k+1}+\mathcal{B}^*z^{k+1}-c)$,
  we obtain that
 \begin{subnumcases}{}
  \mathcal{S}(y^{k}-y^{k+1})\in \partial\vartheta_{\!f}(y^{k+1})+\mathcal{A}\widetilde{x}^{k+1},\nonumber\\
  (\mathcal{T}+\beta\mathcal{B}\mathcal{B}^*)(z^{k}-z^{k+1})\in \partial\varphi_{g}(z^{k+1})+\mathcal{B}\widetilde{x}^{k+1},\nonumber\\
  (\tau\beta)^{-1}(x^{k}-x^{k+1})= c-\mathcal{A}^*y^{k+1}-\mathcal{B}^*z^{k+1},\nonumber
 \end{subnumcases}
 which, by the definitions of $\mathcal{G}$ and $T$, can be written as
 \(
  \mathcal{G}(w^{k}\!-\!w^{k+1}) \in T(\widetilde{w}^{k+1}).
 \)
\end{proof}
 \begin{lemma}\label{lemma2-sequence}
  Let $\{(y^{k},z^{k},x^{k})\}$ be the sequence generated by (\ref{y-sPADMM})-(\ref{x-sPADMM})
  with $\tau\in(0,+\infty)$. For $k=1,2,\ldots$, write $\psi_{k+1}\!:=\!\langle\mathcal{A}^*y^{k+1}\!+\!\mathcal{B}^*z^{k+1}\!-\!c,\mathcal{B}^*(z^{k+1}\!-\!z^{k})\rangle$.
  Then,
  \begin{align*}
  \beta\psi_{k+1}\!+\!\|z^{k+1}\!-\!z^k\|_{\Sigma_{\varphi_{g}}}^2
  &\le \frac{\max(1\!-\!\tau,1\!-\!\tau^{-1})\beta}{2}\|\mathcal{A}^*y^{k}\!+\!\mathcal{B}^*z^{k}\!-\!c\|^2
       \!-\!\frac{1}{2}\|z^{k+1}\!-\!z^{k}\|^2_{\mathcal{T}}\nonumber\\
  &\quad +\frac{\max(1\!-\!\tau,(\tau\!-\!1)\tau)\beta}{2}\|\mathcal{B}^*(z^{k+1}\!-\!z^{k})\|^2
         \!+\!\frac{1}{2}\|z^{k}\!-\!z^{k-1}\|^2_{\mathcal{T}}.
 \end{align*}
 \end{lemma}
 \begin{proof}
  From the optimality of $z^{k}$ and $z^{k+1}$ in equation \eqref{z-sPADMM}, it follows that
 \begin{subnumcases}{}
   0\in \partial\varphi_{g}(z^{k})+\mathcal{B}x^{k}+(1\!-\!\tau)\beta\mathcal{B}(\mathcal{A}^{*}y^{k}+\mathcal{B}^{*}z^{k}-c)
        +\mathcal{T}(z^{k}-z^{k-1}),\nonumber\\
   0\in \partial\varphi_{g}(z^{k+1})+\mathcal{B}x^{k+1}+(1\!-\!\tau)\beta\mathcal{B}(\mathcal{A}^*y^{k+1}+\mathcal{B}^*z^{k+1}\!-\!c)
        +\mathcal{T}(z^{k+1}\!-\!z^{k}).\nonumber
 \end{subnumcases}
  Together with the second inequality in \eqref{fgmonotone} and equation \eqref{x-sPADMM}, we have that
 \begin{align}\label{temp-ineq1}
  \|z^{k+1}\!-\!z^k\|_{\Sigma_{\varphi_{g}}}^2
  &\le (1\!-\!\tau)\beta\langle\mathcal{B}(\mathcal{A}^{*}y^{k+1}+\mathcal{B}^{*}z^{k+1}-c),z^{k}-z^{k+1}\rangle\nonumber\\
  &\quad + \langle\mathcal{B}x^{k+1}\!-\!\mathcal{B}x^{k},z^{k}\!-\!z^{k+1}\rangle
        -\!(1\!-\!\tau)\beta\langle \mathcal{B}(\mathcal{A}^{*}y^{k}\!+\!\mathcal{B}^{*}z^{k}\!-\!c), z^{k}\!-\!z^{k+1}\rangle\nonumber\\
  &\quad  \!-\!\|z^{k+1}\!-\!z^{k}\|^2_{\mathcal{T}}\!-\! \langle \mathcal{T}(z^{k}\!-\!z^{k-1}),z^{k}\!-\!z^{k+1}\rangle\nonumber\\
  &=-\beta\psi_{k+1} \!-\! \langle \mathcal{T}(z^{k}\!-\!z^{k-1}),z^{k}\!-\!z^{k+1}\rangle
     \!-\!\|z^{k+1}\!-\!z^{k}\|^2_{\mathcal{T}}\nonumber\\
   &\quad +(1\!-\!\tau)\beta\langle \mathcal{A}^*y^{k}\!+\!\mathcal{B}^*z^{k}\!-\!c, \mathcal{B}^*(z^{k+1}\!-\!z^{k}) \rangle.
 \end{align}
 Notice that $|\langle \mathcal{T}(z^{k}\!-\!z^{k-1}),z^{k}\!-\!z^{k+1}\rangle|
 \le \frac{1}{2}\|z^{k}\!-\!z^{k-1}\|^2_{\mathcal{T}}+\frac{1}{2}\|z^{k+1}\!-\!z^{k}\|^2_{\mathcal{T}}$ and
 \begin{align*}
  &\ 2(1\!-\!\tau)\langle \mathcal{A}^*y^{k}\!+\!\mathcal{B}^*z^{k}\!-\!c, \mathcal{B}^*(z^{k+1}\!-\!z^{k})\\
  &\le\left\{\begin{array}{ll}
   (1\!-\!\tau)\|\mathcal{A}^*y^{k}\!+\!\mathcal{B}^*z^{k}\!-\!c\|^2+ (1\!-\!\tau)\|\mathcal{B}^*(z^{k+1}\!-\!z^{k})\|^2
   &{\rm if}\ \tau \in (0,1), \nonumber\\
  (1\!-\!\tau^{-1})\|\mathcal{A}^*y^{k}\!+\!\mathcal{B}^*z^{k}\!-\!c\|^2+ (\tau\!-\!1)\tau\|\mathcal{B}(z^{k+1}\!-\!z^{k})\|^2
  &{\rm if}\ \tau \in [1,+\infty).
  \end{array}\right.
 \end{align*}
  Along with equation \eqref{temp-ineq1}, we obtain the desired inequality.
 \end{proof}

 The result of Lemma \ref{lemma2-sequence} is actually implied in the proof of \cite[Theorem B.1]{FPST13}.
 By Lemma \ref{lemma1-sequence} and \ref{lemma2-sequence}, we can establish
  the following key property of the sequence $\{w^k\}$.
 \begin{proposition}\label{Prop1-sequence}
  Let $\{(y^{k},z^{k},x^{k})\}$ be the sequence generated by equation (\ref{y-sPADMM})-(\ref{x-sPADMM})
  with $\tau\in(0,\frac{1+\sqrt{5}}{2})$. Then, for all $k\ge 1$, the following relations hold:
  \begin{subnumcases}{}\label{inexact-a}
   r^{k+1}\in T(\widetilde{w}^{k+1});\\
   \label{inexact-b}
   {D}_{\mathcal{G}}(\widetilde{w}^{k+1},w^{k+1})+\widetilde{\eta}_{k+1}
   \le\sigma_{\!\tau}{D}_{\mathcal{G}}(\widetilde{w}^{k+1},w^{k})+(1\!-\!\gamma_{\tau})\widetilde{\eta}_{k}.
  \end{subnumcases}
  \end{proposition}
  \begin{proof}
  Equation \eqref{inexact-a} directly follows from Lemma \ref{lemma1-sequence} and the definition of $r^{k+1}$.
  We next prove that inequality \eqref{inexact-b} holds. By the definition of $D_{\mathcal{G}}(\cdot,\cdot)$,
  we have that
  \begin{align*}
  D_{\mathcal{G}}(\widetilde{w}^{k+1},w^{k+1})=\frac{1}{2}\|w^{k+1}- \widetilde{w}^{k+1}\|_{\mathcal{G}}^2
  = \frac{1}{2\tau\beta}\|x^{k+1}- \widetilde{x}^{k+1}\|^2,\qquad\qquad\quad\\
  D_{\mathcal{G}}(\widetilde{w}^{k+1},w^{k})
  =\frac{1}{2}\|y^{k+1}\!-\!y^{k}\|^2_{\mathcal{S}}+\frac{1}{2}\|z^{k+1}\!-\!z^{k}\|^2_{(\mathcal{T}+\beta\mathcal{B}\mathcal{B}^*)}
               +\frac{1}{2\tau\beta}\|\widetilde{x}^{k+1}\!-\!x^{k}\|^2.
 \end{align*}
  Multiplying the second equality with $\sigma_{\!\tau}$ and then subtracting the first one yields that
 \begin{align}\label{check-inexact1}
  \sigma_{\!\tau} D_{\mathcal{G}}(\widetilde{w}^{k+1},w^{k})-D_{\mathcal{G}}(\widetilde{w}^{k+1},w^{k+1})
  =&\ \frac{\sigma_{\!\tau}}{2}\|y^{k+1}\!-\!y^{k}\|^2_{\mathcal{S}}
        +\frac{\sigma_{\!\tau}}{2}\|z^{k+1}\!-\!z^{k}\|^2_{(\mathcal{T}+\beta\mathcal{B}\mathcal{B}^*)}\nonumber\\
    &+\frac{\sigma_{\!\tau}}{2\tau\beta}\|\widetilde{x}^{k+1}\!-\!x^{k}\|^2
    -\frac{1}{2\tau\beta}\|x^{k+1}\!-\!\widetilde{x}^{k+1}\|^2.
 \end{align}
 By the definition of $\widetilde{x}^{k+1}$, the last two terms on the right hand side of \eqref{check-inexact1}
 are equal to
 \begin{align*}
 \frac{\sigma_{\!\tau}}{2\tau\beta}\|\widetilde{x}^{k+1}\!-\!x^{k}\|^2
  &= \frac{\beta\sigma_{\!\tau}}{2\tau}\|\mathcal{A}^*y^{k+1}\!+\!\mathcal{B}^*z^{k+1}\!-\!c\|^2
     \!+\!\frac{\beta\sigma_{\!\tau}}{2\tau}\|\mathcal{B}^*(z^{k+1}\!-\!z^{k})\|^2 \!-\!\frac{\beta\sigma_{\!\tau}}{\tau}\psi_{k+1},\\
   \frac{1}{2\tau\beta}\|x^{k+1}\!-\!\widetilde{x}^{k+1}\|^2
  &=\!\frac{(\tau\!-\!1)^2\beta}{2\tau}\|\mathcal{A}^*y^{k+1}\!+\!\mathcal{B}^*z^{k+1}\!-\!c\|^2
    \!+\!\frac{\beta}{2\tau}\|\mathcal{B}^*(z^{k+1}\!-\!z^{k})\|^2\!+\!\frac{(\tau\!-\!1)\beta}{\tau}\psi_{k+1}.
 \end{align*}
 Substituting the two equalities into equation (\ref{check-inexact1}),
 we immediately obtain that
 \begin{align}\label{ADMM-HPE-eq3}
  &\sigma_{\!\tau} D_{\mathcal{G}}(\widetilde{w}^{k+1},w^{k})-D_{\mathcal{G}}(\widetilde{w}^{k+1},w^{k+1})\nonumber\\
  &\!=\frac{(\sigma_{\!\tau}\!-\!(\tau\!-\!1)^2)\beta}{2\tau}\|\mathcal{A}^*y^{k+1}\!+\!\mathcal{B}^*z^{k+1}\!-\!c\|^2
   +\frac{\sigma_{\!\tau}}{2}\|y^{k+1}\!-\!y^{k}\|^2_{\mathcal{S}}+\frac{\sigma_{\!\tau}}{2}\|z^{k+1}\!-\!z^{k}\|^2_{\mathcal{T}}\nonumber\\
  &\quad+\frac{\sigma_{\!\tau}(\tau\!+\!1)\!-\!1}{2\tau}\|z^{k+1}\!-\!z^{k}\|^2_{\beta\mathcal{B}\mathcal{B}^*} \!-\!\frac{\sigma_{\!\tau}\!+\!\tau\!-\!1}{\tau}\beta\psi_{k+1}\nonumber\\
  &\!\geq \frac{(\sigma_{\!\tau}\!-\!(\tau\!-\!1)^2)\beta}{2\tau}\|\mathcal{A}^*y^{k+1}\!+\!\mathcal{B}^*z^{k+1}\!-\!c\|^2
        +\frac{\sigma_{\!\tau}}{2}\|y^{k+1}\!-\!y^{k}\|^2_{\mathcal{S}}\nonumber\\
  &\quad -\frac{(\sigma_{\!\tau}+\tau\!-\!1)\max(1\!-\!\tau,1-\tau^{-1})\beta}{2\tau}\|\mathcal{A}^*y^{k}\!+\!\mathcal{B}^*z^{k}\!-\!c\|^2\\
  &\quad +\frac{\sigma_{\!\tau}(\tau\!+\!1)\!-\!1-(\sigma_{\!\tau}+\tau\!-\!1)\max(1\!-\!\tau,(\tau\!-\!1)\tau)}{2\tau}
        \|z^{k+1}-z^{k}\|^2_{\beta\mathcal{B}\mathcal{B}^*}\nonumber\\
  &\quad +\frac{\nu_{\tau}}{2}\|z^{k+1}\!-\!z^{k}\|^2_{\mathcal{T}}
       -\frac{\sigma_{\!\tau}+\tau\!-\!1}{2\tau}\|z^{k}\!-\!z^{k-1}\|^2_{\mathcal{T}}
       +\frac{\sigma_{\!\tau}+\tau\!-\!1}{\tau}\|z^{k+1}\!-\!z^{k}\|^2_{\Sigma_{\varphi_{g}}}\nonumber\\
  &\!=\widetilde{\eta}_{k+1}-\!\frac{(\sigma_{\!\tau}\!+\!\tau\!-\!1)\max(1\!-\!\tau,1\!-\!\tau^{-1})\beta}{2\tau}\|\mathcal{A}^*y^{k}\!+\!\mathcal{B}^*z^{k}\!-\!c\|^2
    \!-\!\frac{\sigma_{\!\tau}\!+\!\tau\!-\!1}{2\tau}\|z^{k}\!-\!z^{k-1}\|^2_{\mathcal{T}}\nonumber
 \end{align}
  where the inequality is using $\sigma\!+\!\tau\!-\!1 \!>\! 0$ and Lemma \ref{lemma2-sequence},
  and the second equality is by \eqref{imply-sigma} and the definition of $\widetilde{\eta}_{k+1}$.
  We proceed the arguments by two cases.

 \medskip
 \noindent
 {\bf Case 1:} $\tau \in (0,1)$. In this case, let $\gamma_1= \frac{\tau\sigma_{\!\tau}}{\sigma_{\!\tau}-(\tau-1)^2}$ and $\gamma_2=\frac{\tau\sigma_{\!\tau}}{\nu_{\tau}}$. Then we have that
  $(\sigma_{\!\tau}+\tau\!-\!1)(1\!-\!\tau)=(1\!-\!\gamma_1)(\sigma_{\!\tau}\!-\!(\tau\!-\!1)^2)$
  and $(\sigma_{\!\tau}+\tau\!-\!1)=(1\!-\!\gamma_2) \nu_{\tau}$, and consequently,
 \begin{align*}
 &\ \frac{(\sigma_{\!\tau}+\tau\!-\!1)(1\!-\!\tau)\beta}{2\tau}\|\mathcal{A}^*y^{k}\!+\!\mathcal{B}^*z^{k}\!-\!c\|^2 +\frac{\sigma_{\!\tau}+\tau\!-\!1}{2\tau}\|z^{k}\!-\!z^{k-1}\|^2_{\mathcal{T}}\nonumber\\
 &= \frac{(1\!-\!\gamma_1)(\sigma_{\!\tau}\!-\!(\tau\!-\!1)^2)\beta}{2\tau}\|\mathcal{A}^*y^{k}\!+\!\mathcal{B}^*z^{k}\!-\!c\|^2 +\frac{(1\!-\!\gamma_2)\nu_{\tau}}{2\tau}\|z^{k}\!-\!z^{k-1}\|^2_{\mathcal{T}}.
 \end{align*}
  Substituting this equality into inequality (\ref{ADMM-HPE-eq3}) and noting that $\tau\in(0,1)$, we obtain that
 \begin{align*}
 &\ \sigma_{\!\tau}D_{\mathcal{G}}(\widetilde{w}^{k+1},w^k)-D_{\mathcal{G}}(w^{k+1},\widetilde{w}^{k+1})\nonumber\\
 &\ge \widetilde{\eta}_{k+1}-\frac{(1\!-\!\gamma_1)(\sigma_{\!\tau}\!-\!(\tau\!-\!1)^2)\beta}{2\tau}
      \|\mathcal{A}^*y^{k}\!+\!\mathcal{B}^*z^{k}\!-\!c\|^2
    -\frac{(1\!-\!\gamma_2)\nu_{\!\tau}}{2\tau}\|z^{k}\!-\!z^{k-1}\|^2_{\mathcal{T}}\nonumber\\
 & \geq \widetilde{\eta}_{k+1}-(1\!-\!\gamma_2)\widetilde{\eta}_{k}=\widetilde{\eta}_{k+1}-(1-\gamma_{\tau})\widetilde{\eta}_{k},
 \end{align*}
 where the second inequality is due to $1>\gamma_1\ge\gamma_2>0$ implied by equation \eqref{sigma-h1h2-equa1} and
 \begin{equation}\label{etak-ineq}
 \widetilde{\eta}_{k}\!\ge\!\frac{(\sigma_{\!\tau}-(\tau\!-\!1)^2)\beta}{2\tau}\|\mathcal{A}^*y^{k}\!+\!\mathcal{B}^*z^{k}\!-\!c\|^2 +\frac{\nu_{\!\tau}}{2\tau}\|z^{k}\!-\!z^{k-1}\|^2_{\mathcal{T}}.
 \end{equation}
 and the last one is due to $\gamma_{\tau}=\gamma_2$ implied by the definition of $\gamma_{\tau}$ and equation \eqref{sigma-h1h2-equa2}.

 \medskip
 \noindent
 {\bf Case 2:} $\tau \in [1,\frac{1+\sqrt{5}}{2})$.
 Let $\gamma_1=\frac{\sigma_{\!\tau}-(\tau+1)(\tau-1)^2}{\tau(\sigma_{\!\tau}-(\tau-1)^2)}$
 and $\gamma_2= \frac{\tau\sigma_{\!\tau}}{\nu_{\tau}}$. Then we have that
 $(\sigma_{\!\tau}\!+\!\tau\!-\!1)(1\!-\!\tau^{-1})=(1\!-\!\gamma_1)(\sigma_{\!\tau}\!-\!(\tau\!-\!1)^2)$
 and $(\sigma_{\!\tau}\!+\!\tau\!-\!1)=(1\!-\!\gamma_2)\nu_{\tau}$, which implies that
 \begin{align*}
  &\ \frac{(\sigma_{\!\tau}+\tau\!-\!1)(1\!-\!\tau^{-1})\beta}{2\tau}\|\mathcal{A}^*y^{k}\!+\!\mathcal{B}^*z^{k}\!-\!c\|^2 +\frac{\sigma_{\!\tau}+\tau\!-\!1}{2\tau}\|z^{k}\!-\!z^{k-1}\|^2_{\mathcal{T}}\nonumber\\
  &= \frac{(1\!-\!\gamma_1)(\sigma_{\!\tau}\!-\!(\tau\!-\!1)^2)\beta}{2\tau}\|\mathcal{A}^*y^{k}\!+\!\mathcal{B}^*z^{k}\!-\!c\|^2 +\frac{(1\!-\!\gamma_2)\nu_{\!\tau}}{2\tau}\|z^{k}\!-\!z^{k-1}\|^2_{\mathcal{T}}.
 \end{align*}
 Substituting this equality into inequality (\ref{ADMM-HPE-eq3}) and noting that $\tau\in[1,\frac{1+\sqrt{5}}{2})$ yields that
\begin{align*}
 &\ \sigma_{\!\tau} D_{\mathcal{G}}(\widetilde{w}^{k+1},w^{k})-D_{\mathcal{G}}(w^{k+1},\widetilde{w}^{k+1})\nonumber\\
 &\ge \widetilde{\eta}_{k+1}-\frac{(1\!-\!\gamma_1)(\sigma_{\!\tau}\!-\!(\tau\!-\!1)^2)\beta}{2\tau}
      \|\mathcal{A}^*y^{k}\!+\!\mathcal{B}^*z^{k}\!-\!c\|^2
    -\frac{(1\!-\!\gamma_2)\nu_{\!\tau}}{2\tau}\|z^{k}\!-\!z^{k-1}\|^2_{\mathcal{T}}\nonumber\\
 & \ge \widetilde{\eta}_{k+1}-\max\{1\!-\!\gamma_1, 1\!-\!\gamma_2\}\widetilde{\eta}_{k}
    =\widetilde{\eta}_{k+1}-(1-\gamma_{\tau})\widetilde{\eta}_{k},
\end{align*}
 where the second inequality is due to inequality \eqref{etak-ineq} and
 $\gamma_1,\gamma_2\in(0,1)$ implied by \eqref{sigma-h1h2-equa1} and \eqref{sigma-h1h2-equa3},
 and the last one is by the definition of $\gamma_{\tau}$ and equation \eqref{sigma-h1h2-equa3}.
 \end{proof}
 \begin{remark}\label{remark-sequence}
 If $\mathcal{G}$ is positive definite, then \eqref{inexact-a}-\eqref{inexact-b}
 can be equivalently written as
 \begin{subnumcases}{}\label{inexact-aa}
  \!r^{k+1}\in\! T(\widetilde{w}^{k+1}),\\
   \label{inexact-bb}
  \frac{1}{2}\big[\|r^{k+1}\!+\!\mathcal{G}(\widetilde{w}^{k+1}\!-\!w^{k})\|_{\mathcal{G}^{-1}}\big]^2\!+\widetilde{\eta}_{k+1}
  \le\!\frac{\sigma_{\!\tau}}{2}\big[\|\mathcal{G}(\widetilde{w}^{k+1}\!-\!w^{k})
  \|_{\mathcal{G}^{-1}}\big]^2\!+\!(1\!-\!\gamma_{\tau})\widetilde{\eta}_{k},\quad \\
  \label{inexact-cc}
  w^{k+1}=w^{k}-\mathcal{G}^{-1}r^{k+1},
 \end{subnumcases}
  where equation \eqref{inexact-aa}-\eqref{inexact-bb} is solving approximately
  the positive definite proximal subproblem $0\in T(\cdot)+\mathcal{G}(\cdot-w^{k-1})$ with the error $\frac{1}{2}\big[\|r^{k+1}\!+\!\mathcal{G}(\widetilde{w}^{k+1}\!-\!w^{k})\|_{\mathcal{G}^{-1}}\big]^2$
  satisfying \eqref{inexact-bb}, and \eqref{inexact-cc} is a scaled extra-gradient step
  to guarantee the convergence of $\{w^k\}$. Motivated by this, we call
  \eqref{inexact-a}-\eqref{inexact-b} a generalized HPE iteration formula.
  However, it should be emphasized that since $\mathcal{G}$ here is only positive semidefinite,
  the above iteration formula neither belongs to the HPE method \cite{SS99,SS01}
  nor falls into the non-Euclidean HPE framework studied in \cite{GMM2016}
  since the operator $S$ there can not degenerate to zero.
 \end{remark}

  By Proposition \ref{Prop1-sequence}, we have the following two results of the sequence $\{(y^k,z^k,x^k)\}$.
 \begin{proposition}\label{Prop2-sequence}
  Let $\{(y^{k},z^{k},x^{k})\}$ be the sequence generated by equation (\ref{y-sPADMM})-(\ref{x-sPADMM})
  with $\tau\in(0,\frac{1+\sqrt{5}}{2})$. Then, for any $w\in\mathbb{H}$,
  the following inequality holds for all $k\ge 1$:
  \begin{equation}\label{contractive-eq1}
   D_{\mathcal{G}}(w,w^{k})\!-\!{D}_{\mathcal{G}}(w,w^{k+1})+(1\!-\!\gamma_{\tau})\widetilde{\eta}_{k}
   \ge (1\!-\!\sigma_{\!\tau}){D}_{\mathcal{G}}(\widetilde{w}^{k+1},w^{k})
   +\langle r^{k+1},\widetilde{w}^{k+1}\!-\!w\rangle+\widetilde{\eta}_{k+1}.
  \end{equation}
  In particular, for any $w^*\in{T}^{-1}(0)$, the following inequality holds for $k\ge 1$:
  \begin{equation} \label{contractive-eq2}
   {D}_{\mathcal{M}}(w^*,w^{k+1})+\eta_{k+1} \le {D}_{\mathcal{M}}(w^*,w^{k})-
   (1\!-\!\sigma_{\!\tau}){D}_{\mathcal{G}}(\widetilde{w}^{k+1},w^{k})+(1\!-\!\gamma_{\tau})\eta_{k}.
  \end{equation}
 \end{proposition}
 \begin{proof}
  From the definition of $D_{\mathcal{G}}(\cdot,\cdot)$ and inequality \eqref{inexact-b},
  it immediately follows that
  \begin{align*}
   &\ {D}_{\mathcal{G}}(w,w^{k})\!-\!{D}_{\mathcal{G}}(w,w^{k+1})+(1\!-\!\gamma_{\tau})\widetilde{\eta}_{k}
   +\sigma_{\!\tau}{D}_{\mathcal{G}}(\widetilde{w}^{k+1},w^{k})\\
   &\ge {D}_{\mathcal{G}}(w,w^{k})\!-\!{D}_{\mathcal{G}}(w,w^{k+1})
        +{D}_{\mathcal{G}}(\widetilde{w}^{k+1},w^{k+1})+\widetilde{\eta}_{k+1}\\
   &={D}_{\mathcal{G}}(\widetilde{w}^{k+1},w^{k})+\langle r^{k+1},\widetilde{w}^{k+1}\!-\!w\rangle
      +\widetilde{\eta}_{k+1},
  \end{align*}
  which implies that inequality \eqref{contractive-eq1} holds.
  Taking $w=w^*\in T^{-1}(0)$ in \eqref{contractive-eq1}, we have
  \[
   D_{\mathcal{G}}(w^*,w^{k})\!-\!{D}_{\mathcal{G}}(w^*,w^{k+1})+(1\!-\!\gamma_{\tau})\widetilde{\eta}_{k}
   \ge (1\!-\!\sigma_{\!\tau}){D}_{\mathcal{G}}(\widetilde{w}^{k+1},w^{k})
   +\langle r^{k+1},\widetilde{w}^{k+1}\!-\!w^*\rangle+\widetilde{\eta}_{k+1}.
  \]
  Notice that  $0\in T(w^*)$ and $\mathcal{G}(w^{k}\!-\!w^{k+1}) \in T(\widetilde{w}^{k+1})$.
  It follows from \eqref{Tmonotone} that
  \[
    \langle r^{k+1},\widetilde{w}^{k+1}\!-\!w^*\rangle=\langle\mathcal{G}(w^{k}\!-\!w^{k+1}),\widetilde{w}^{k+1}\!-\!w^*\rangle
    \ge \|\widetilde{w}^{k+1}-w^*\|_{\Sigma}^2= \|w^{k+1}-w^*\|_{\Sigma}^2.
  \]
  By recalling that $\mathcal{M}=\mathcal{G}+\Sigma$, the last two inequalities imply that
  \begin{align*}
  &\  D_{\mathcal{M}}(w^*,w^{k+1})+\widetilde{\eta}_{k+1}\nonumber\\
  &\le D_{\mathcal{G}}(w^*,w^{k})-(1\!-\!\sigma_{\!\tau}){D}_{\mathcal{G}}(\widetilde{w}^{k+1},w^{k})
      \!+\!(1\!-\!\gamma_{\tau})\widetilde{\eta}_{k}-\frac{1}{2}\|w^{k+1}\!-\!w^*\|_{\Sigma}^2\\
   &\le D_{\mathcal{M}}(w^*,w^{k})-(1\!-\!\sigma_{\!\tau}){D}_{\mathcal{G}}(\widetilde{w}^{k+1},w^{k})
       \!+\!(1\!-\!\gamma_{\tau})\widetilde{\eta}_{k}\!-\frac{1}{4}\|w^{k+1}\!-\!w^k\|_{\Sigma}^2,
  \end{align*}
  where the last inequality is due to $\|w^{k+1}\!-\!w^*\|_{\Sigma}^2\ge \frac{1}{2}\|w^{k+1}\!-\!w^k\|_{\Sigma}^2 -\|w^{k}\!-\!w^*\|_{\Sigma}^2$.
  Recalling that $\eta_{k+1}=\widetilde{\eta}_{k+1}+\frac{1}{4}\|w^{k+1}\!-\!w^k\|_{\Sigma}^2$,
  from the last inequality we obtain that
  \[
    D_{\mathcal{M}}(w^*,w^{k+1})+\eta_{k+1}
    \le D_{\mathcal{M}}(w^*,w^{k})-(1\!-\!\sigma_{\!\tau}){D}_{\mathcal{G}}(\widetilde{w}^{k+1},w^{k})
        +(1\!-\!\gamma_{\tau})\widetilde{\eta}_{k},
  \]
  which along with $\widetilde{\eta}_{k}\le \eta_{k}$ for all $k\ge 1$ implies
  the desired inequality \eqref{contractive-eq2}.
 \end{proof}
 \begin{proposition}\label{ergodic-sequence}
  Let $\{(y^{k},z^{k},x^{k})\}$ be given by (\ref{y-sPADMM})-(\ref{x-sPADMM})
  with $\tau\in(0,\frac{1+\sqrt{5}}{2})$. For each $k\ge 1$, let $\alpha_i \ge 0$ for $i=1,2,\ldots,k\!+\!1$
  be such that $\alpha_1\!=0$ and $\sum_{i=1}^{k}\alpha_{i+1}\!=1$ and define
 \begin{equation*}
  \overline{w}^{k}:=\sum_{i=1}^{k}\alpha_{i+1}\widetilde{w}^{i+1},\
  \overline{r}^{k}:=\sum_{i=1}^{k}\alpha_{i+1}r^{i+1}\ {\rm and}\
  \overline{\epsilon}^{k}:= \sum_{i=1}^{k}\alpha_{i+1}\langle \widetilde{w}^{i+1}\!-\!\overline{w}^{k},r^{i+1}\!-\!\overline{r}^{k}\rangle.
 \end{equation*}
  Then, the following relations hold for all $k\ge 1$:
 \begin{subnumcases}{}
 \overline{r}^{k} \in T^{[\overline{\epsilon}^{k}]}(\overline{w}^{k})\ \ {\rm with}\ \overline{\epsilon}^{k}\ge 0,\\
 \label{rkbound}
   \|\overline{r}^{k}\|
  \le\big(2\sqrt{\|\mathcal{G}\|}\sqrt{d^*\!+\!d_1\!+\!\eta_{1}}+\|\mathcal{G}(\widehat{w}^{*})\|\big)
     \left(\textstyle\sum_{i=1}^{k}\big|\alpha_{i}\!-\!\alpha_{i+1}\big|\!+\!\alpha_{k+1}\right),\\
  \label{epskbound}
  \overline{\epsilon}^{k}
   \le 2(d^*\!+\!d_1\!+\!\eta_1)\frac{(11-9\sigma_{\!\tau})}{1-\sigma_{\!\tau}}\sum_{i=1}^{k}|\alpha_{i+1}\!-\!\alpha_{i}|
 \end{subnumcases}
 where $d^*:=\min_{w\in T^{-1}(0)}{D}_{\mathcal{M}}(w,w^{0})={D}_{\mathcal{M}}(\widehat{w}^*,w^{0})$,
 $d_1:={D}_{\mathcal{M}}(w^1,w^{0})$, and
  $T^{[\varepsilon]}(w):=\big\{v\in\mathbb{H}\ |\ \langle v'\!-\!v,w'\!-\!w\rangle\ge -\varepsilon\ \ {\rm for\ all}\
    w'\in\mathbb{H},v'\in T(w')\big\}$ is the $\varepsilon$-enlargement of $T$ at $w$.
 \end{proposition}
 \begin{proof}
  For $k\ge 1$, we know from equation \eqref{inexact-a} that
  $r^{i}\in T(\widetilde{w}^i)$ for $i=2,\ldots,k+1$.
  By the maximal monotonicity of $T$, using the weak transportation formula
  \cite[Theorem 2.3]{BSS99} yields that $\overline{\epsilon}^k\ge 0$ and
  $\overline{r}^{k}\in T^{[\overline{\epsilon}^{k}]}(\overline{w}^{k})$ for $k\ge 1$.
  We next establish inequality \eqref{rkbound}.
  From the definition of $r^{i+1}$, it follows that
 \begin{align}\label{weight-complexity-eq1}
  \|\overline{r}^{k}\|
  &\le\big\|{\textstyle\sum_{i=1}^{k}}\big(\alpha_{i+1}\mathcal{G}(w^{i+1})
       \!-\!\alpha_{i}\mathcal{G}(w^{i})\big)\big\|
       \!+\!\big\|{\textstyle\sum_{i=1}^{k}}\big(\alpha_{i}\!-\!\alpha_{i+1}\big)\mathcal{G}(w^{i})\big\|\nonumber\\
  &=\big\|\alpha_{k+1}\mathcal{G}(w^{k+1})\big\|
     \!+\!\big\|{\textstyle\sum_{i=1}^{k}}\big(\alpha_{i}\!-\!\alpha_{i+1}\big)\mathcal{G}(w^{i})\big\|.
 \end{align}
  From (\ref{contractive-eq2}), it follows that
 ${D}_{\mathcal{M}}(\widehat{w}^*,w^{k})\le{D}_{\mathcal{M}}(\widehat{w}^*,w^{1})+\eta_{1}$ for $k\ge 1$.
 So, we have that
 \begin{align}\label{weight-complexity-eq2}
  \|\alpha_{k+1}\mathcal{G}(w^{k+1})\|
  &\le \alpha_{k+1}\|\mathcal{G}(w^{k+1})\!-\!\mathcal{G}(\widehat{w}^*)\|
       +\alpha_{k+1}\|\mathcal{G}(\widehat{w}^*)\| \nonumber\\
  &\le \alpha_{k+1}{\sqrt{2\|\mathcal{G}\|}}\sqrt{{D}_{\mathcal{G}}(\widehat{w}^*,w^{k+1})}
        \!+\alpha_{k+1}\|\mathcal{G}(\widehat{w}^*)\| \nonumber\\
  &\le \big(\sqrt{2\|\mathcal{G}\|}\sqrt{{D}_{\mathcal{M}}(\widehat{w}^*,w^{1})+\eta_{1}}+\|\mathcal{G}(\widehat{w}^*)\|\big)
       \alpha_{k+1}.
 \end{align}
 Using the inequality ${D}_{\mathcal{M}}(\widehat{w}^*,w^{k})\le{D}_{\mathcal{M}}(\widehat{w}^*,w^{1})+\eta_{1}$ for $k\ge 1$ again,
 we obtain that
 \begin{align*}
  \big\|{\textstyle\sum_{i=1}^{k}}\big(\alpha_{i}\!-\!\alpha_{i+1}\big)\mathcal{G}(w^{i})\big\|
  &\le\textstyle\sum_{i=1}^{k}\big[|\alpha_{i}\!-\!\alpha_{i+1}|\|\mathcal{G}(w^{i}\!-\!\widehat{w}^*)\|
          \!+\!|\alpha_{i}\!-\!\alpha_{i+1}|\|\mathcal{G}(\widehat{w}^*)\|\big]\nonumber\\
  &\le \sum_{i=1}^{k}\Big[|\alpha_{i}\!-\!\alpha_{i+1}|\sqrt{2\|\mathcal{G}\|}\sqrt{{D}_{\mathcal{G}}(\widehat{w}^*,w^{i})}
         \!+\!|\alpha_{i}\!-\!\alpha_{i+1}|\|\mathcal{G}(\widehat{w}^*)\|\Big]\nonumber\\
  &\le \Big(\!\sqrt{2\|\mathcal{G}\|}\sqrt{{D}_{\mathcal{M}}(\widehat{w}^*,w^{1})\!+\!\eta_{1}}+\!\|\mathcal{G}(\widehat{w}^*)\|\Big)
       {\textstyle\sum_{i=1}^{k}}\big|\alpha_{i}\!-\!\alpha_{i+1}\big|.
 \end{align*}
 Combining this with inequalities (\ref{weight-complexity-eq2}) and (\ref{weight-complexity-eq1})
 and noting that ${D}_{\mathcal{M}}(\widehat{w}^*,w^{1})\le 2d^*+2d_1$,
 we obtain the desired inequality \eqref{rkbound}.

 \medskip

Next we show that inequality \eqref{epskbound} holds.
 From the definition of $\overline{\epsilon}^{k}$, it follows that
 \begin{equation}\label{weight-complexity-eq5}
  \overline{\epsilon}^{k}
   =\textstyle \sum_{i=1}^{k} \alpha_{i+1}\langle \widetilde{w}^{i+1}\!-\!\overline{w}^{k},r^{i+1}\rangle.
 \end{equation}
 Notice that $r^{i+1}=\mathcal{G}(w^{i}\!-\!w^{i+1})$ by \eqref{inexact-a}.
 After an elementary calculation, we have that
 \begin{align}\label{equality}
    \langle \widetilde{w}^{i+1}\!-\!\overline{w}^{k},r^{i+1}\rangle
    ={D}_{\mathcal{G}}(\widetilde{w}^{i+1},w^{i+1})\!-\!{D}_{\mathcal{G}}(\widetilde{w}^{i+1},w^{i})
    +{D}_{\mathcal{G}}(\overline{w}^{k},w^{i})\!-\!{D}_{\mathcal{G}}(\overline{w}^{k},w^{i+1}).
 \end{align}
 Since $D_{\mathcal{G}}(\widetilde{w}^{i+1},w^{i})\!-\!{D}_{\mathcal{G}}(\widetilde{w}^{i+1},w^{i+1})+\widetilde{\eta}_{i}-\widetilde{\eta}_{i+1}\ge 0$
 implied by \eqref{contractive-eq1}, it follows that
 \[
 \sum_{i=1}^{k}\alpha_{i+1}\big[{D}_{\mathcal{G}}(\widetilde{w}^{i+1},w^{i+1})-D_{\mathcal{G}}(\widetilde{w}^{i+1},w^{i})\big]
  \le \sum_{i=1}^{k}\alpha_{i+1}\big[\widetilde{\eta}_{i}-\widetilde{\eta}_{i+1}\big]
  \le\sum_{i=1}^{k}(\alpha_{i+1}-\alpha_i)\widetilde{\eta}_{i}.
 \]
 Combining this inequality with equality \eqref{equality} and inequality \eqref{weight-complexity-eq5}, we obtain that
 \begin{align}\label{weight-complexity-eq9}
  &\ \overline{\epsilon}^{k}
  =\sum_{i=1}^{k}\alpha_{i+1}\big[{D}_{\mathcal{G}}(\widetilde{w}^{i+1},w^{i+1})-D_{\mathcal{G}}(\widetilde{w}^{i+1},w^{i})\big]
  +\sum_{i=1}^{k}\alpha_{i+1}\big[{D}_{\mathcal{G}}(\overline{w}^{k},w^{i})\!-\!{D}_{\mathcal{G}}(\overline{w}^{k},w^{i+1})\big]
  \nonumber\\
  &\le \sum_{i=1}^{k}(\alpha_{i+1}-\alpha_i)\widetilde{\eta}_{i}+
  \sum_{i=1}^{k}\left[(\alpha_{i+1}\!-\!\alpha_{i}){D}_{\mathcal{G}}(\overline{w}^{k},w^{i})
                                       \!+\!\alpha_{i}{D}_{\mathcal{G}}(\overline{w}^{k},w^{i})
                                     \!-\!\alpha_{i+1}\!{D}_{\mathcal{G}}(\overline{w}^{k},w^{i+1})\right]\nonumber\\
  &\le\sum_{i=1}^{k}(\alpha_{i+1}\!-\!\alpha_i)\widetilde{\eta}_{i}
      +\sum_{i=1}^{k}(\alpha_{i+1}\!-\!\alpha_i){D}_{\mathcal{G}}(\overline{w}^{k},w^{i})
  \le\sum_{i=1}^{k}|\alpha_{i+1}\!-\!\alpha_i|\big[\widetilde{\eta}_{i}+{D}_{\mathcal{G}}(\overline{w}^{k},w^{i})].
 \end{align}
 By the convexity of ${D}_{\mathcal{G}}(\cdot,w^{i})$ and the definition of $\overline{w}^{k}$,
 it follows that for $i=1,2,\ldots,k$,
 \begin{align*}
  &\ {D}_{\mathcal{G}}(\overline{w}^{k},w^{i})
  \le\sum_{j=1}^{k}\alpha_{j+1}D_{\mathcal{G}}(\widetilde{w}^{j+1},w^{i})
  \le 2\sum_{j=1}^{k}\alpha_{j+1}\big({D}_{\mathcal{G}}(\widetilde{w}^{j+1},w^{j})+{D}_{\mathcal{G}}(w^{j},w^{i})\big) \nonumber\\
  &\le \frac{2(D_{\mathcal{G}}(\widehat{w}^*,w^{1})+\eta_{1})}{1\!-\!\sigma_{\!\tau}}
        +2\sum_{j=1}^{k}\alpha_{j+1}{D}_{\mathcal{G}}(w^{j},w^{i})\nonumber\\
  &\le \frac{2(d^*\!+\!d_1\!+\!\eta_{1})}{1\!-\!\sigma_{\!\tau}}
       +4\sum_{j=1}^{k}\alpha_{j+1}\big({D}_{\mathcal{G}}(\widehat{w}^*,w^{i})\!+\!\mathcal{D}_{\mathcal{G}}(\widehat{w}^*,w^{j})\big)
  \le \frac{20\!-\!16\sigma_{\!\tau}}{1\!-\!\sigma_{\!\tau}}\big(d^*\!+\!d_1\!+\!\eta_{1}\big)
 \end{align*}
 where the third inequality is since
 $(1\!-\!\sigma_{\!\tau})D_{\mathcal{G}}(\widetilde{w}^{j+1},w^{j})\le D_{\mathcal{G}}(\widehat{w}^*,w^{1})+\eta_{1}$
 for $j\ge 1$ by \eqref{contractive-eq2}, and the last one is using
 ${D}_{\mathcal{G}}(\widehat{w}^*,w^{i})\le{D}_{\mathcal{G}}(\widehat{w}^*,w^{1})+\eta_{1}
 \le 2d^*\!+\!2d_1\!+\!\eta_{1}$ for $i\ge 1$.
 Substituting the last inequality into  (\ref{weight-complexity-eq9})
 and using $\widetilde{\eta}_{i}\le\eta_{i}\le {D}_{\mathcal{G}}(\widehat{w}^*,w^{1})+\eta_{1}$ for $i\ge 1$
 and $D_{\mathcal{G}}(\overline{w}^{k},w^{1})\le \frac{20-16\sigma_{\!\tau}}{1-\sigma_{\!\tau}}\big(d^*\!+\!d_1\!+\!\eta_{1}\big)$
 implied by the last inequality, we obtain that
 \begin{align*}
  \overline{\epsilon}^{k}
  &\le\sum_{i=1}^{k}|\alpha_{i+1}\!-\!\alpha_{i}|
  \Big[\frac{20\!-\!16\sigma_{\!\tau}}{1\!-\!\sigma_{\!\tau}}\big(d^*\!+\!d_1\!+\!\eta_{1}\big)+2(d^*\!+\!d_1\!+\!\eta_1)\Big]
    \nonumber \\
  &\le 2(d^*\!+\!d_1\!+\!\eta_1)\frac{(11-9\sigma_{\!\tau})}{1-\sigma_{\!\tau}}\sum_{i=1}^{k}|\alpha_{i+1}\!-\!\alpha_{i}|
 \end{align*}
 Consequently, we obtain the desired bound for $\overline{\epsilon}^{k}$. The proof is completed
 \end{proof}

 By using Proposition \ref{ergodic-sequence}, we can establish the following main result of this paper.
 \begin{theorem}\label{ADMM-ergodic}
  Let $\{(y^{k},z^{k},x^{k})\}$ be given by (\ref{y-sPADMM})-(\ref{x-sPADMM}) with $\tau\in(0,\frac{1+\sqrt{5}}{2})$.
  For each $k\ge 1$, let $\alpha_i \ge 0$ for $i=1,2,\ldots,k\!+\!1$
  be such that $\alpha_1=0$ and $\sum_{i=1}^{k}\alpha_{i+1}=1$ and define
  \begin{align*}
  \overline{y}^{k}\!=\!\sum_{i=1}^{k}\alpha_{i+1}{y}^{i+1},\ \
  \overline{z}^{k}=\sum_{i=1}^{k}\alpha_{i+1}z^{i+1},\ \
   \overline{x}^{k}\!=\!\sum_{i=1}^{k}\alpha_{i+1}\widetilde{x}^{i+1},\qquad\qquad\\
  \overline{\epsilon}_{y}^{k}\!=\!\sum_{i=1}^{k}\alpha_{i+1}\langle {y}^{i+1}\!-\!\overline{y}^{k},\delta_{y}^{i+1}\!-\!\mathcal{A}\widetilde{x}^{i+1}\rangle
  \ \ {\rm with}\  \delta_{y}^{i+1}=\mathcal{S}(y^{i}\!-\!y^{i+1}),\qquad\nonumber\\
  \overline{\epsilon}_{z}^{k}=\sum_{i=1}^{k}\alpha_{i+1}\langle {z}^{i+1}\!-\!\overline{z}^{k},\delta_{z}^{i+1}
  \!-\!\mathcal{B}\widetilde{x}^{i+1}\rangle
  \ \ {\rm with}\  \delta_{z}^{i+1}\!=(\mathcal{T}\!+\!\beta\mathcal{B}^*\mathcal{B})(z^{i}\!-\!z^{i+1}).
 \end{align*}
  Then, for each $k\ge 1$, we have $\overline{\epsilon}_{y}^{k}\ge 0$, $\overline{\epsilon}_{z}^{k}\ge 0$,
  and the following two inequalities
  \begin{subnumcases}{}\label{epskbound}
  \overline{\epsilon}_{y}^{k}+\overline{\epsilon}_{z}^{k}
  \le 2(d^*\!+\!d_1\!+\!\eta_1)\frac{(11-9\sigma_{\!\tau})}{1-\sigma_{\!\tau}}\sum_{i=1}^{k}|\alpha_{i+1}\!-\!\alpha_{i}|,\\
  {\rm dist}^2\big(0,\partial_{\overline{\epsilon}_{y}^{k}}\vartheta_{\!f}(\overline{y}^{k})\!+\!\mathcal{A}\overline{x}^{k}\big)
   +{\rm dist}^2\big(0,\partial_{\overline{\epsilon}_{z}^{k}} \varphi_{g}(\overline{z}^{k})\!+\!\mathcal{B}\overline{x}^{k}\big)
  \!+\!\|\mathcal{A}^*\overline{y}^{k}\!+\!\mathcal{B}^*\overline{z}^{k}\!-\!c\|^2\nonumber\\
  \label{KKT-bound}
  \le\Big(2\sqrt{\|\mathcal{G}\|(d^*\!+\!d_1\!+\!\eta_1)}+\|\mathcal{G}(\widehat{w}^{*})\|\Big)^2
     \Big({\textstyle\sum_{i=1}^{k}}\big|\alpha_{i+1}\!-\!\alpha_{i}\big|\!+\!\alpha_{k+1}\Big)^2,
  \end{subnumcases}
  where, for a given $\epsilon\ge 0$, $\partial_{\epsilon}\vartheta_{\!f}(y)$ denotes
  the $\epsilon$-subdifferential of the function $\vartheta_{\!f}$ at $y$.
 \end{theorem}
 \begin{proof}
  From \eqref{inexact-a}, it follows that $r^{i+1}=\mathcal{G}(w^{i}\!-\!w^{i+1})\in T(\widetilde{w}^{i+1})$
  for each $i\ge 1$. Write
  \[
   \overline{r}^{k}=(\overline{r}_1^{k},\overline{r}_2^{k},\overline{r}_3^{k})
   \!:=\!{\textstyle\sum_{i=1}^{k}}\alpha_{i+1}r^{i+1}\ \ {\rm and}\ \  \overline{\epsilon}^k:={\textstyle\sum_{i=1}^{k}}\alpha_{i+1}\langle \widetilde{w}^{i+1}\!-\!\overline{w}^{k},r^{i+1}\rangle.
  \]
  Then, by using Proposition \ref{ergodic-sequence}, we obtain that
  \(
    \overline{r}^{k}\in T^{[\overline{\epsilon}^{k}]}(\overline{y}^{k},\overline{z}^{k},\overline{x}^{k})
  \)
  with
  \begin{align}\label{rkbound1}
   \|\overline{r}^{k}\|
  \le\big(2\sqrt{\|\mathcal{G}\|}\sqrt{d^*\!+\!d_1\!+\!\eta_{1}}+\|\mathcal{G}(\widehat{w}^{*})\|\big)
     \left(\textstyle\sum_{i=1}^{k}\big|\alpha_{i}\!-\!\alpha_{i+1}\big|\!+\!\alpha_{k+1}\right),\\
   \label{epskbound1}
  \overline{\epsilon}^{k}
   \le 2(d^*\!+\!d_1\!+\!\eta_1)\frac{(11-9\sigma_{\!\tau})}{1-\sigma_{\!\tau}}\sum_{i=1}^{k}|\alpha_{i+1}\!-\!\alpha_{i}|.
   \qquad\qquad\qquad
  \end{align}
  Notice that $r_1^{i+1}\!-\!\mathcal{A}\widetilde{x}^{i+1}\in\partial\vartheta_{\!f}(y^{i+1})$ and
  $r_2^{i+1}\!-\!\mathcal{B}\widetilde{x}^{i+1}\in\partial\varphi_{g}(z^{i+1})$ for each $i\ge 1$
  by the definition of $\mathcal{G}$ and $r^{i+1}\in T(\widetilde{w}^{i+1})$.
  From \cite[Theorem 2.3]{BSS99}, it follows that
  \(
    \overline{r}^{k}_{1}-\mathcal{A}\overline{x}^{k}\in
  \partial_{\overline{\epsilon}_{y}^{k}}\vartheta_{\!f}(\overline{y}^{k})
  \)
  and
  \(
    \overline{r}^{k}_{2}-\mathcal{B}\overline{x}^{k}
  \in\partial_{\overline{\epsilon}_{z}^{k}}\varphi_{g}(\overline{z}^{k})
  \)
  with $\overline{\epsilon}_{y}^{k},\overline{\epsilon}_{z}^{k}\ge 0$.
  Notice that $\overline{r}_3^k=\mathcal{A}^*\overline{y}^{k}\!+\!\mathcal{B}^*\overline{z}^{k}\!-\!c$
  since $r^{i+1}_{3}=(\tau\beta)^{-1}(x^{i}\!-\!x^{i+1})
  =c\!-\!\mathcal{A}^*y^{i+1}\!-\!\mathcal{B}^*z^{i+1}$. Therefore, it holds that
  \[
     {\rm dist}^2\big(0,\partial_{\overline{\epsilon}_{y}^{k}}\vartheta_{\!f}(\overline{y}^{k})\!+\!\mathcal{A}\overline{x}^{k}\big)
   +{\rm dist}^2\big(0,\partial_{\overline{\epsilon}_{z}^{k}}\varphi_{g}(\overline{z}^{k})\!+\!\mathcal{B}\overline{x}^{k}\big)
  \!+\!\|\mathcal{A}^*\overline{y}^{k}\!+\!\mathcal{B}^*\overline{z}^{k}\!-\!c\|^2
  \le\|\overline{r}^k\|^2.
  \]
  Together with \eqref{rkbound1}, we obtain \eqref{KKT-bound}.
  By the definitions of $\overline{\epsilon}_{y}^{k}$ and $\overline{\epsilon}_{z}^{k}$,
  we have that
  \begin{align*}
    &\overline{\epsilon}_{y}^{k}+\overline{\epsilon}_{z}^{k}= \sum_{i=1}^{k}\alpha_{i+1}
     \big(\langle {y}^{i+1}\!-\!\overline{y}^{k},\delta_{y}^{i+1}-\mathcal{A}\widetilde{x}^{i+1}\rangle
      +\langle{z}^{i+1}\!-\!\overline{z}^{k},\delta_{z}^{i+1}-\mathcal{B}\widetilde{x}^{i+1}\rangle\big)\nonumber\\
   &\!=\sum_{i=1}^{k}\alpha_{i+1}
     \big(\langle{y}^{i+1}\!-\!\overline{y}^{k},r^{i+1}_{1}\rangle\!+\!\langle{z}^{i+1}\!-\!\overline{z}^{k},r^{i+1}_{2}\rangle
     \!-\!\langle \mathcal{A}^*({y}^{i+1}\!-\!\overline{y}^{k})\!+\!\mathcal{B}^*({z}^{i+1}\!-\!\overline{z}^{k}),
            \widetilde{x}^{i+1}\rangle\big)\nonumber\\
   &\!=\sum_{i=1}^{k}\alpha_{i+1}
     \big(\langle \widetilde{w}^{i+1}\!-\!\overline{w}^{k},r^{i+1}\rangle
 \!-\!\langle \widetilde{x}^{i+1}\!-\!\overline{x}^{k},r^{i+1}_{3}\rangle \!-\!\langle \mathcal{A}^*({y}^{i+1}\!-\!\overline{y}^{k})\!+\!\mathcal{B}^*({z}^{i+1}\!-\!\overline{z}^{k}),\widetilde{x}^{i+1}\rangle\big).
 \end{align*}
  Recall that $r^{i+1}_{3}=(\tau\beta)^{-1}(x^{i}\!-\!x^{i+1})
  =c\!-\!\mathcal{A}^*y^{i+1}\!-\!\mathcal{B}^*z^{i+1}$. Then, we obtain that
 \begin{align*}
  \overline{\epsilon}_{y}^{k}+ \overline{\epsilon}_{z}^{k}
  &=\overline{\epsilon}^{k} -\sum_{i=1}^{k}\alpha_{i+1}
    \left(\langle \widetilde{x}^{i+1}\!-\!\overline{x}^{k},r^{i+1}_{3}\rangle +\langle \mathcal{A}^*{y}^{i+1}\!+\!\mathcal{B}^*{z}^{i+1}\!-\!\mathcal{A}^*\overline{y}^{k}\!-\!\mathcal{B}^*\overline{z}^{k},
    \widetilde{x}^{i+1}\rangle\right)\nonumber\\
  &=\overline{\epsilon}^{k} -\sum_{i=1}^{k}\alpha_{i+1}
   \left(\big\langle \overline{x}^{k},\mathcal{A}^*{y}^{i+1}\!+\!\mathcal{B}^*{z}^{i+1}\!-\!c\big\rangle
    +\big\langle c\!-\!\mathcal{A}^*\overline{y}^{k}\!-\!\mathcal{B}^*\overline{z}^{k},\widetilde{x}^{i+1}\big\rangle\right)\nonumber\\
  &=\overline{\epsilon}^{k}\!-\!\big\langle \overline{x}^{k},\mathcal{A}^*\overline{y}^{k}\!+\!\mathcal{B}^*\overline{z}^{k}\!-\!c\big\rangle
     \!-\!\big\langle c\!-\!\mathcal{A}^*\overline{y}^{k}\!-\!\mathcal{B}^*\overline{z}^{k},\overline{x}^{k}\big\rangle
     =\overline{\epsilon}^{k}.
 \end{align*}
  Along with the nonnegativity of $\overline{\epsilon}_{y}^{k}$ and $\overline{\epsilon}_{z}^{k}$
  and \eqref{epskbound1}, we obtain inequality \eqref{epskbound}.
 \end{proof}
  \begin{remark}
  {\bf (a)} Theorem \ref{ADMM-ergodic} provides a weighted iteration complexity on
  the KKT residuals for the sPADMM with $\tau\in(0,\frac{1+\sqrt{5}}{2})$.
  Among others, the iteration complexity bound depends on the choice of $\alpha_i$.
  When $\alpha_i\equiv\frac{1}{k}$ or $\alpha_i=\frac{i}{\sum_{i=1}^k(i+1)}$
  for $i=2,\ldots,k\!+\!1$, we have $\overline{\epsilon}_{y}^{k}+\overline{\epsilon}_{z}^{k}\le \frac{M}{k}$
  with $M=4(d^*\!+\!d_1\!+\!\eta_1)\frac{11-9\sigma_{\!\tau}}{1-\sigma_{\!\tau}}$
  and
 \[
 {\rm dist}^2\big(0,\partial_{\overline{\epsilon}_{y}^{k}}\vartheta_{\!f}(\overline{y}^{k})\!+\!\mathcal{A}\overline{x}^{k}\big)
   +{\rm dist}^2\big(0,\partial_{\overline{\epsilon}_{z}^{k}} \varphi_{g}(\overline{z}^{k})\!+\!\mathcal{B}\overline{x}^{k}\big)
  \!+\!\|\mathcal{A}^*\overline{y}^{k}\!+\!\mathcal{B}^*\overline{z}^{k}\!-\!c\|^2
  \le\frac{C}{k^2}
  \]
  with $C=64\big(\sqrt{\|\mathcal{G}\|(d^*\!+\!d_1\!+\!\eta_1)}+\frac{1}{2}\|\mathcal{G}(\widehat{w}^{*})\|\big)^2$.
  The constants $M$ and $C$ depend on the spectral norm of $\mathcal{G}$ (and then
  those of the proximal operators $\mathcal{S}$ and $\mathcal{T}$),
  the distance from the initial point $w^0$ to $T^{-1}(0)$ (i.e., the value of $d^*$),
  the distance from $w^0$ to $w^1$ (i.e., the value of $d_1$)
  and the value of $\eta_1$. Clearly, if the spectral norms of $\mathcal{S}$ and $\mathcal{T}$
  are smaller, then the complexity bound is better; and if $d^*$, $d_1$ and $\eta_1$ are smaller,
  the complexity bound is better. As will be shown in the next section,
  the values of $d_1$ and $\eta_1$ can be controlled by the distance
  from $w^0$ to $T^{-1}(0)$. Thus, the complexity bound mainly depends on
  the spectral norms of $\mathcal{S}$ and $\mathcal{T}$ and the value of $d^*$.

  \medskip
  \noindent
  {\bf (b)} Theorem \ref{ADMM-ergodic} provides an $\mathcal{O}(1/k)$ weighted
  iteration complexity on the KKT residuals yielded by the sGS based semi-proximal ADMM
  proposed in \cite{LST161} for the convex composite quadratic programming involving
  two nonsmooth and multiple linear/quadratic convex functions, since this type of ADMM
  was shown to be equivalent to a semi-proximal ADMM for \eqref{prob}.
  It also provides an $\mathcal{O}(1/k)$ weighted iteration
  complexity on the KKT residuals for the classic ADMM and the positive definite
  proximal ADMM with a large step-size.

  \medskip
  \noindent
  {\bf(c)} The iteration complexity of Theorem \ref{ADMM-ergodic} was derived without
  requiring $\mathcal{S}$ and $\mathcal{T}$ to be such that
  $\mathcal{S}\!+\!\Sigma_{\vartheta_{\!f}}\!+\!\beta\mathcal{A}\mathcal{A}^*$
  and $\mathcal{T}\!+\!\Sigma_{\varphi_{g}}\!+\!\beta\mathcal{B}\mathcal{B}^*$ are positive
  definite as required in \cite{FPST13,HSZ15} for the convergence and the linear convergence rate
  of the semi-proximal ADMM.
 \end{remark}
 \section{Upper estimations of $d_1$ and $\eta_1$}\label{sec3}

  In this section, we present an upper estimation for $d_1$ and $\eta_1$ involving
  in the iteration complexity bounds in \eqref{epskbound} and \eqref{KKT-bound}.
  For this purpose, we let $\mathcal{H}_1,\mathcal{H}_2,\mathcal{H}_3:\mathbb{H}\to\mathbb{H}$
  be the self-adjoint block diagonal positive semidefinite linear operators defined by
  \begin{align*}
   \mathcal{H}_1:={\rm Diag}\left[\begin{matrix}
                    2(\mathcal{S}\!+\Sigma_{\vartheta_{\!f}}\!+\!4\beta\mathcal{A}\mathcal{A}^*),\
                    7\beta\mathcal{B}\mathcal{B}^*,\ 7\beta^{-1}
    \end{matrix}\right],\qquad\qquad\\
    \mathcal{H}_2:={\rm Diag}\left[\begin{matrix}
                    28(\mathcal{S}\!+\Sigma_{\vartheta_{\!f}}\!+\!5\beta\mathcal{A}\mathcal{A}^*),\
                    2(\mathcal{T}\!+\Sigma_{\varphi_{g}}\!+\!53\beta\mathcal{B}\mathcal{B}^*),\ 105\beta^{-1}
    \end{matrix}\right],\quad\\
    \mathcal{H}_3:=4\tau{\rm Diag}\left[\begin{matrix}
                   30(\mathcal{S}\!+\Sigma_{\vartheta_{\!f}}\!+\!5\beta\mathcal{A}\mathcal{A}^*),\
                    2(\mathcal{T}\!+\Sigma_{\varphi_{g}}\!+\!57\beta\mathcal{B}\mathcal{B}^*),\ 112\beta^{-1}
    \end{matrix}\right].
  \end{align*}
  The following lemma implies a relation for $\|w^0-w^1\|_{\mathcal{M}}^2$ and
  $\sum_{i=1}^3{\rm dist}_{\mathcal{H}_i}^2(w^0,T^{-1}(0))$.
  \begin{lemma}\label{distance-x1x0}
   The first iteration point $(y^1,z^1,x^1)$ in equation \eqref{y-sPADMM}-\eqref{x-sPADMM} satisfies
   \begin{subnumcases}{}\label{ypart}
   \|y^1\!-\!y^0\|^2_{\mathcal{S}\!+\Sigma_{\vartheta_{\!f}}\!+\!\beta\mathcal{A}\mathcal{A}^*}
   \le {\rm dist}_{\mathcal{H}_1}^2(w^0,T^{-1}(0)),\\
   \label{zpart}
   \|z^1\!-\!z^0\|^2_{\mathcal{T}\!+\!\beta\mathcal{B}\mathcal{B}^*+\Sigma_{\varphi_{g}}}
   \le {\rm dist}_{\mathcal{H}_2}^2(w^0,T^{-1}(0)),\\
   \label{xpart}
   (\tau\beta)^{-1}\|x^1\!-\!x^0\|^2\le {\rm dist}_{\mathcal{H}_3}^2(w^0,T^{-1}(0)).
   \end{subnumcases}
  \end{lemma}
  \begin{proof}
   The optimal condition of the minimization problem \eqref{y-sPADMM} with $k=0$ at $y^1$ is
   \begin{equation*}
   0\in \partial\vartheta_{\!f}(y^1)+\mathcal{A}x^0+\beta\mathcal{A}(\mathcal{A}^*y^1+\mathcal{B}^*z^0-c)+\mathcal{S}(y^1-y^0).
   \end{equation*}
   Let $w^*$ be an arbitrary point from $T^{-1}(0)$. Since $0\in T(w^*)$, we have
   \(
     -\mathcal{A}x^*\in\partial\vartheta_{\!f}(y^*).
   \)
   Using the first inequality in \eqref{fgmonotone} and
   $\|y^0\!-\!y^1\|_{\Sigma_{\vartheta_{\!f}}}^2\!\le\!2\|y^*\!-\!y^1\|_{\Sigma_{\vartheta_{\!f}}}^2
   \!+\!2\|y^*\!-\!y^0\|_{\Sigma_{\vartheta_{\!f}}}^2$
   yields that
  \begin{align*}
   &\langle \mathcal{A}(x^0\!-\!x^*),y^*\!-\!y^1\rangle
   \!+\!\beta \langle \mathcal{A}(\mathcal{A}^*y^1\!+\!\mathcal{B}^*z^0\!-\!c),y^*\!-\!y^1\rangle
   \!+\!\langle \mathcal{S}(y^1\!-\!y^0),y^*\!-\!y^1\rangle \nonumber\\
   &\ge \|y^*\!-\!y^1\|_{\Sigma_{\vartheta_{\!f}}}^2\ge \frac{1}{2}\|y^0\!-\!y^1\|_{\Sigma_{\vartheta_{\!f}}}^2-\|y^*\!-\!y^0\|_{\Sigma_{\vartheta_{\!f}}}^2.
  \end{align*}
   Notice that $\mathcal{A}^*y^1\!+\!\mathcal{B}^*z^0\!-\!c
   = \mathcal{A}^*(y^1\!-\!y^0)+\mathcal{A}^*(y^0\!-\!y^*)\!+\!\mathcal{B}^*(z^0\!-\!z^*)$. We have that
   \begin{align}\label{y-upper-temp1}
   \|y^1\!-\!y^0\|^2_{\mathcal{S}\!+\!\beta\mathcal{A}^*\mathcal{A}+\frac{1}{2}\Sigma_{\vartheta_{\!f}}}
   &\le \langle \beta\mathcal{A}^*(y^0\!-\!y^*)+\beta\mathcal{B}^*(z^0\!-\!z^*)+x^0\!-\!x^*,\mathcal{A}^*(y^*\!-\!y^0)\rangle\nonumber\\
   &\quad +\langle \beta\mathcal{A}^*(y^0\!-\!y^*)+\beta\mathcal{B}^*(z^0\!-\!z^*)+x^0\!-\!x^*,\mathcal{A}^*(y^0\!-\!y^1)\rangle\nonumber\\
   &\quad +\langle(\mathcal{S}\!+\!\beta\mathcal{A}\mathcal{A}^*)(y^1\!-\!y^0),y^*\!-\!y^0\rangle
         +\|y^*\!-\!y^0\|_{\Sigma_{\vartheta_{\!f}}}^2.
  \end{align}
  For the first term on the right hand side of \eqref{y-upper-temp1},
  by using $\langle\mathcal{B}^*(z^0\!-\!z^*),\mathcal{A}^*(y^*\!-\!y^0)\rangle
   \leq \frac{1}{2}\|z^0\!-\!z^*\|^2_{\mathcal{B}\mathcal{B}^*}
     +\frac{1}{2}\|y^0\!-\!y^*\|^2_{\mathcal{A}\mathcal{A}^*}$ and
  $\langle x^0\!-\!x^*,\mathcal{A}^*(y^*\!-\!y^0)\rangle
   \leq \frac{1}{2\beta}\|x^0\!-\!x^*\|^2 +\frac{\beta}{2}\|y^0\!-\!y^*\|^2_{\mathcal{A}\mathcal{A}^*}$,
  \begin{equation}\label{y-term1}
    \langle \beta\mathcal{A}^*(y^0\!-\!y^*)+\beta\mathcal{B}^*(z^0\!-\!z^*)+x^0\!-\!x^*,\mathcal{A}^*(y^*\!-\!y^0)\rangle
    \le\frac{1}{2}\|z^0\!-\!z^*\|^2_{\beta\mathcal{B}^*\mathcal{B}}\!+\!\frac{1}{2\beta}\|x^0\!-\!x^*\|^2.
  \end{equation}
  For the second term, using the Cauchy-Schwartz inequality yields that
  \begin{align}\label{y-term2}
  &\langle \beta\mathcal{A}^*(y^0\!-\!y^*)+\beta\mathcal{B}^*(z^0\!-\!z^*)+x^0\!-\!x^*,\mathcal{A}^*(y^0\!-\!y^1)\rangle\nonumber\\
  &\leq \frac{\beta}{4}\|y^0\!-\!y^1\|^2_{\mathcal{A}^*\mathcal{A}}
      +\frac{1}{\beta}\|\beta\mathcal{A}^*(y^0\!-\!y^*)+\beta\mathcal{B}^*(z^0\!-\!z^*)+x^0\!-\!x^*\|^2\nonumber\\
  &\leq \frac{1}{4}\|y^0\!-\!y^1\|^2_{\mathcal{S}+\beta\mathcal{A}\mathcal{A}^*}
      +3\|y^0\!-\!y^*\|^2_{\beta\mathcal{A}\mathcal{A}^*}+3\|z^0\!-\!z^*\|^2_{\beta\mathcal{B}\mathcal{B}^*}
      +\frac{3}{\beta}\|x^0\!-\!x^*\|^2.
  \end{align}
  For the third term on the right hand side of \eqref{y-upper-temp1},
  using the same technique yields that
  \begin{align}\label{y-term3}
  \langle(\mathcal{S}\!+\!\beta\mathcal{A}\mathcal{A}^*)(y^1\!-\!y^0),y^*\!-\!y^0\rangle
  \leq \frac{1}{4}\|y^1\!-\!y^0\|_{\mathcal{S}\!+\!\beta\mathcal{A}\mathcal{A}^*}^2
       +\|y^*\!-\!y^0\|_{\mathcal{S}\!+\!\beta\mathcal{A}\mathcal{A}^*}^2,
  \end{align}
  Substituting inequalities \eqref{y-term1}-\eqref{y-term3} to equation \eqref{y-upper-temp1},
  we obtain that
  \begin{equation}\label{ypart1}
   \|y^1\!-\!{y}^0\|^2_{\mathcal{S}+\Sigma_{\vartheta_{\!f}}+\beta\mathcal{A}\mathcal{A}^*}
    \le 2\|y^0\!-\!y^*\|_{\mathcal{S}\!+\Sigma_{\vartheta_{\!f}}\!+\!4\beta\mathcal{A}\mathcal{A}^*}^2
         +7\|z^0\!-\!z^*\|^2_{\beta\mathcal{B}\mathcal{B}^*}+7\beta^{-1}\|x^0\!-\!x^*\|^2.
  \end{equation}
  This, by the definition of $\mathcal{H}_1$, is equivalent to
  $\|y^1\!-\!y^0\|^2_{\mathcal{S}\!+\Sigma_{\vartheta_{\!f}}\!+\!\beta\mathcal{A}\mathcal{A}^*}
   \le \|w^0\!-\!w^*\|_{\mathcal{H}_1}^2$.

  \medskip

  Using the same arguments as above for the problem \eqref{z-sPADMM} with $k=0$, we obtain that
  \begin{align*}\label{z-upper-temp2}
   \|z^1\!-\!z^0\|^2_{\mathcal{T}\!+\!\beta\mathcal{B}\mathcal{B}^*+\Sigma_{\varphi_{g}}}
   &\le 2\|z^0\!-\!z^*\|_{\mathcal{T}\!+\Sigma_{\varphi_{g}}\!+\!4\beta\mathcal{B}\mathcal{B}^*}^2
         +7\|y^1\!-\!y^*\|^2_{\beta\mathcal{A}\mathcal{A}^*}+\frac{7}{\beta}\|x^0\!-\!x^*\|^2.
  \end{align*}
  Notice that $7\|y^1\!-\!y^*\|^2_{\beta\mathcal{A}\mathcal{A}^*}
  \!\le\! 14\|y^1\!-\!y^0\|^2_{\beta\mathcal{A}\mathcal{A}^*}+14\|y^0\!-\!y^*\|^2_{\beta\mathcal{A}\mathcal{A}^*}$.
  Using \eqref{ypart1} yields that
  \[
   7\|y^1\!-\!y^*\|^2_{\beta\mathcal{A}\mathcal{A}^*}
   \le 28\|y^0\!-\!y^*\|_{\mathcal{S}\!+\Sigma_{\vartheta_{\!f}}\!+\!4\beta\mathcal{A}\mathcal{A}^*}^2
         +98\|z^0\!-\!z^*\|^2_{\beta\mathcal{B}\mathcal{B}^*}+\frac{98}{\beta}\|x^0\!-\!x^*\|^2
         +14\|y^0\!-\!y^*\|^2_{\beta\mathcal{A}\mathcal{A}^*}.
  \]
  The last two inequalities implies
  $\|z^1\!-\!z^0\|^2_{\mathcal{T}\!+\!\beta\mathcal{B}\mathcal{B}^*+\Sigma_{\varphi_{g}}}
   \le \|w^0\!-\!w^*\|_{\mathcal{H}_2}^2$.
  From \eqref{x-sPADMM}, it follows that
  \(
   \|x^1\!-\!x^0\|^2
  \le 4\tau^2\beta^2\big(\|\mathcal{A}^*(y^1\!-\!y^0)\|^2\!+\!\|\mathcal{A}^*(y^0\!-\!y^*)\|^2
       \!+\!\|\mathcal{B}^*(z^1\!-\!z^0)\|^2\!+\!\|\mathcal{B}^*(z^0\!-\!z^*)\|^2\big).
  \)
  Together with the upper estimations for
  $\|y^1\!-\!y^0\|^2_{\mathcal{S}\!+\Sigma_{\vartheta_{\!f}}\!+\!\beta\mathcal{A}\mathcal{A}^*}$
  and $\|z^1\!-\!z^0\|^2_{\mathcal{T}\!+\!\beta\mathcal{B}\mathcal{B}^*+\Sigma_{\varphi_{g}}}$,
  a simple calculation yields that $(\tau\beta)^{-1}\|x^1\!-\!x^0\|^2\le \|w^0\!-\!w^*\|_{\mathcal{H}_3}^2$.
  Thus, for any $w^*\in T^{-1}(0)$,
  \begin{subnumcases}{}\label{ypart}
   \|y^1\!-\!y^0\|^2_{\mathcal{S}\!+\Sigma_{\vartheta_{\!f}}\!+\!\beta\mathcal{A}\mathcal{A}^*}
   \le \|w^0\!-\!w^*\|_{\mathcal{H}_1}^2,\nonumber\\
   \label{zpart}
   \|z^1\!-\!z^0\|^2_{\mathcal{T}\!+\!\beta\mathcal{B}\mathcal{B}^*+\Sigma_{\varphi_{g}}}
   \le \|w^0\!-\!w^*\|_{\mathcal{H}_2}^2,\nonumber\\
   \label{xpart}
   (\tau\beta)^{-1}\|x^1\!-\!x^0\|^2\le \|w^0\!-\!w^*\|_{\mathcal{H}_3}^2.\nonumber
   \end{subnumcases}
  By the arbitrariness of $w^*$, we obtain the desired result.
 \end{proof}

  Now we are in a position to establish the upper estimation for $d_1$ and $\eta_1$.
 \begin{proposition}\label{estimation-d1-eta1}
  For $d_1$ and $\eta_1$ in equation \eqref{epskbound}-\eqref{KKT-bound}, the following inequalities hold:
  \begin{align*}
   d_1 \le\frac{1}{2}\sum_{i=1}^3{\rm dist}_{\mathcal{H}_i}^2(w^0,T^{-1}(0))\ \ {\rm and}\ \
   \eta_{1}\le\max\Big(\frac{1}{2\tau^2},2\Big)\sum_{i=1}^3{\rm dist}_{\mathcal{H}_i}^2(w^0,T^{-1}(0)).
  \end{align*}
  \end{proposition}
  \begin{proof}
  By the definition of $d_1$ and Lemma \ref{distance-x1x0}, we have the following inequality
  \begin{align*}
   d_1=D_{\mathcal{M}}({w}^0,w^1)
   &=\frac{1}{2}\|y^1\!-\!y^0\|_{\mathcal{S}+\Sigma_{\vartheta_{\!f}}}^2
    +\frac{1}{2}\|z^1\!-\!z^0\|^2_{\mathcal{T}+\Sigma_{\varphi_{g}}+\beta\mathcal{B}\mathcal{B}^*}
    +\frac{1}{2\tau\beta}\|x^1\!-\!x^0\|^2\\
   &\le \frac{1}{2}\sum_{i=1}^3{\rm dist}_{\mathcal{H}_i}^2(w^0,T^{-1}(0)).
  \end{align*}
  From $x^1-x^0=\tau\beta(\mathcal{A}^*y^{1}\!+\!\mathcal{B}^*z^{1}\!-\!c)$,
  the definition of $\eta_1$ and Lemma \ref{distance-x1x0}, it follows that
  \begin{align*}
   \eta_{1}
   &=\frac{(\sigma\!-\!(\tau\!-\!1)^2)}{2\tau^3\beta}\|x^1\!-\!x^0\|^2+\frac{\sigma}{2}\|y^1\!-\!y^0\|^2_{\mathcal{S}}
              +\frac{h(\tau,\sigma)}{2}\|z^{1}\!-\!z^{0}\|^2_{\mathcal{T}+\beta\mathcal{B}^*\mathcal{B}}
               +\frac{1}{4}\|w^1\!-\!w^0\|^2_{\Sigma}\nonumber\\
   &\le \frac{1}{2\tau^3\beta}\|x^1\!-\!x^0\|^2+\frac{1}{2}\|y^1\!-\!y^0\|^2_{\mathcal{S}}
              +\tau\|z^{1}\!-\!z^{0}\|^2_{\mathcal{T}+\beta\mathcal{B}^*\mathcal{B}}
               +\frac{1}{4}\|w^1\!-\!w^0\|^2_{\Sigma}\\
   &\leq \frac{1}{2\tau^2}{\rm dist}_{\mathcal{H}_3}^2(w^0,T^{-1}(0))+\frac{1}{2}{\rm dist}_{\mathcal{H}_1}^2(w^0,T^{-1}(0))
         +2{\rm dist}_{\mathcal{H}_2}^2(w^0,T^{-1}(0)),\\
   &\le\max\Big(\frac{1}{2\tau^2},2\Big)\sum_{i=1}^3{\rm dist}_{\mathcal{H}_i}^2(w^0,T^{-1}(0)).
  \end{align*}
  The last two equations give the desired results. The proof is completed.
  \end{proof}

  Proposition \ref{estimation-d1-eta1} shows that $d_1$ and $\eta_1$ can be controlled
  by the distance from the initial point $w^0$ to the solution set $T^{-1}(0)$.
  The derivation of this upper estimation does not require $\mathcal{S}$ and $\mathcal{T}$
  to be such that $\mathcal{S}\!+\!\Sigma_{\vartheta_{\!f}}\!+\!\beta\mathcal{A}\mathcal{A}^*$
  and $\mathcal{T}\!+\!\Sigma_{\varphi_{g}}\!+\!\beta\mathcal{B}\mathcal{B}^*$ are positive
  definite.

 \bigskip
 \noindent
 {\large\bf Acknowledgements.} The authors would like to thank Professor Defeng Sun
 from National University of Singapore for the discussions on
 the weighted iteration complexity of the semi-proximal ADMM and his helpful comments
 on the revision of this paper.

 \end{document}